\documentclass[amssymb,amsfonts,final]{amsart}
\usepackage[intlimits,sumlimits,noDcommand,narrowiints]{kpfonts}
\usepackage[english]{babel}
\usepackage{mathtools}
\usepackage{comment}
\usepackage{slashed}

\usepackage[notref,notcite]{showkeys}

\usepackage[pdfusetitle]{hyperref}

\theoremstyle{plain}
\newtheorem{thm}{Theorem}[section]

\newtheorem{prop}[thm]{Proposition}
\newtheorem{lem}[thm]{Lemma}

\theoremstyle{definition}

\theoremstyle{remark}
\newtheorem{rmk}[thm]{Remark}

\newcommand*\R{\mathbb{R}}
\newcommand*\D{\mathrm{d}}

\DeclarePairedDelimiterX\set[2]\lbrace\rbrace{#1\;\delimsize\vert\;#2}
\newcommand*\eqdef{\overset{\mbox{\tiny{def}}}{=}}

\newcommand*\mr{\mathring}

\DeclareMathOperator{\dvol}{dvol}

\DeclareMathOperator{\supp}{supp}
\declareslashed{}{/}{.1}{0}{\xi}
\newcommand*\Lebw{\mathcal{L}}
\newcommand*\Sobw{\mathcal{W}}

\numberwithin{equation}{section}
 
\begin{document}
\title[Global GNS inequalities and applications]{Global versions of Gagliardo-Nirenberg-Sobolev inequality and applications to wave and Klein-Gordon equations}
\author{Leonardo Abbrescia}
\address{Michigan State University, East Lansing, Michigan, USA}
\email{abbresci@msu.edu}
\author{Willie Wai Yeung Wong}
\address{Department of Mathematics, Michigan State University, East Lansing, Michigan, USA}
\email{wongwwy@math.msu.edu}
\thanks{WWY Wong is supported by a Collaboration Grant from the Simons Foundation, \#585199.}

\begin{abstract}
	We prove global, or space-time weighted, versions of the Gagliardo-Nirenberg interpolation inequality, with $L^p$ ($p < \infty$) endpoint, adapted to a hyperboloidal foliation. The corresponding versions with $L^\infty$ endpoint was first introduced by Klainerman and is the basis of the classical vector field method, which is now one of the standard techniques for studying long-time behavior of nonlinear evolution equations. 
	We were motivated in our pursuit by settings where the vector field method is applied to an energy hierarchy with growing higher order energies. 
	In these settings the use of the $L^p$ endpoint versions of Sobolev inequalities can allow one to gain essentially one derivative in the estimates, which would then give a corresponding gain of decay rate.
	The paper closes with the analysis of one such model problem, where our new estimates provide an improvement. 
\end{abstract}

\maketitle
\tableofcontents

\section{Introduction} \label{sec:intro}
We are led to the subject of the present manuscript, which are weighted $L^2$--$L^p$ type Sobolev estimates adapted to hyperboloidal foliations, through our previous work on the stability of travelling wave solutions to the membrane equation \cite{Planewaves}. 
A feature of our argument is the use of an \emph{energy hierarchy}, where higher-order energies that control the higher order derivatives of the unknown with respect to space-time weighted vector fields are allowed to grow, in time, with rate of growth depending on the number of derivatives taken. 
Such a hierarchy appears necessary due to the large (in fact infinite energy) background solution causing the equations for higher order weighted derivatives to have coefficients that are themselves growing in time. While we were able to successfully study the problem there for all spatial dimensions $d \geq 3$, the case with $d = 2$ eluded our analysis. 

The difficulty, as we understood it, stems from the interaction of the global Sobolev inequalities with the energy hierarcy. 
The standard argument, using the energy method, for either the stability problem or the local existence problem for quasilinear waves, handles the nonlinearities with the general prescription of ``putting the highest order derivative factor in $L^2$ and the remainder in $L^\infty$.'' The $L^\infty$ term is then controlled by a \emph{higher order} $L^2$ integral using some version of the Sobolev inequality.

The use of the $L^2$--$L^\infty$ Sobolev inequality naturally introduces some amount of inefficiency.
A poignant example occurs in dimension $d = 2$. Using only the $L^\infty$ type Sobolev estimates we can bound
\[ \| u^2 \|_{L^2(\R^2)} \leq \|u\|_{L^\infty(\R^2)} \|u\|_{L^2(\R^2)} \lesssim \|u\|_{H^2(\R^2)} \|u\|_{L^2(\R^2)}.\]
(Scaling would have given us the first factor of $u$ in $H^1$, but as we know the end-point Sobolev embedding in $L^\infty$ is false.)
Using $L^p$ type Sobolev inequalities instead we can appeal to Ladyzhenskaya's inequality to get
\[ \|u^2\|_{L^2(\R^2)} \lesssim \|u\|_{H^1(\R^2)} \|u\|_{L^2(\R^2)}\]
for a gain of one derivative. (Or rather, one should think of this as the $L^\infty$ Sobolev inequality \emph{losing} one derivative.)

For classical applications where all orders of energies are typically bounded, this derivative loss is of no consequence, except in the need of working with higher regularity initial data. In \cite{Planewaves}, however, higher energies are allowed to \emph{grow}. This type of derivative loss will then be accompanied by a loss of decay of the solution, which can severely impact whether the estimates are closable, especially in the even dimensions. 
Exactly such a difficulty seems to be happening when we tried extending our analysis in \cite{Planewaves} from the case of spatial dimension $d \geq 3$ to the case of spatial dimension $d = 2$.
The main difficulty arises in the analysis of the quasilinear terms; we will not discuss precisely this difficulty in the present paper, in view of other technical complications for dealing with quasilinear equations. 
At the end of this paper, we will however give a flavor of the improvements one can obtain by showing how the use of $L^2$--$L^p$ type, global Gagliardo-Nirenberg-Sobolev inequalities can improve the analysis of a $d = 2$ semilinear model problem.

\section{The global GNS inequalities} \label{sec:GGNS}
The goal of this section is to develop certain weighted Gagliardo-Nirenberg-Sobolev (GNS) inequalities. 
These inequalities can be considered as being adapted to suitably weighted energy integrals associated to studying the linear wave and Klein-Gordon equations using a hyperboloidal foliation. 
The Morrey versions of these inequalities, which give $L^\infty$ control based on $L^2$ integrals of higher derivatives, have been previously described in \cite{LeFloch} and \cite{Wong2017}. 
Our results can be viewed as the counterpart to this theory extended to weighted $L^p$ based Sobolev spaces. 

Keeping in mind the expectation that these integrals will be viewed as being adapted to a hyperboloidal foliation, we will set our notation accordingly. 
By $\Sigma_\tau$ we refer to the hyperboloid in $\R^{1+d}$ given by 
\begin{equation}
	\Sigma_\tau \eqdef \{t^2 - |x|^2 = \tau^2, t > 0\}.
\end{equation}
We can parametrize it by $\R^d$ via the map
\begin{equation}
	(x^1, \ldots, x^d) \mapsto ( t =  \sqrt{\tau^2 + |x|^2}, x^1, \ldots, x^d) \in \R^{1+d}.
\end{equation}
For convenience throughout we will denote by
\begin{equation}
	w_\tau(x) \eqdef \sqrt{\tau^2 + |x|^2}, \quad x\in \R^d.
\end{equation}
We note that the value of $w_\tau$, when thinking of $\Sigma_\tau$ as embedded in $\R^{1+d}$, of course agrees with the value of the $t$ coordinate; we use the notation $w_\tau$ as mental aid to work intrinsically on $\Sigma_\tau$ whenever appropriate. 

The Minkowski metric on $\R^{1+d}$ induces a Riemannian metric on $\Sigma_{\tau}$, which is given by the matrix-valued function 
\begin{equation}
	g_{ij} = \delta_{ij} - \frac{x^i x^j}{w_\tau(x)^2}
\end{equation}
relative to the parametrization above. This being a rank-1 perturbation of the Euclidean metric, the corresponding volume form can be easily computed to be
\begin{equation}\label{eq:volformweight}
	\dvol = \frac{\tau}{w_\tau} \D x^1 \wedge \cdots \wedge \D x^d.
\end{equation}

The Minkowski space $\R^{1+d}$ admits as Killing vector fields the Lorentzian boosts, given as
\begin{equation}
	L^i \eqdef x^i \partial_t + t \partial_{x^i}.
\end{equation}
These vector fields are \emph{tangent} to the hypersurfaces $\Sigma_\tau$ for every $\tau > 0$, and in the parametrization above can be identified with 
\begin{equation}
	L^i \cong w_\tau \partial_{x^i}.
\end{equation}
We remark that 
\[ L^i w_\tau = x^i, \quad L^i x^i = w_\tau.\]
In particular, we have that for any string of derivatives
\begin{equation}\label{eq:weightder}
\left\lvert L^{i_1} \cdots L^{i_K} w_\tau \right\rvert \leq w_\tau.
\end{equation}

\subsection{The basic global GNS inequalities} 

The Nirenberg argument \cite{Nirenberg59} is built upon the fundamental theorem of calculus. Given a point $x\in \R^d$, we will write
\[ x'_i(s) \eqdef (x^1, x^2, \ldots, x^{i-1}, s, x^{i+1}, \ldots, x^d) \]
as the point where the $i$th coordinate of $x$ is replaced by the real parameter $s$. Then the fundamental theorem of calculus states that, for any smooth, compactly supported function $u$,
\begin{equation}
	|u(x)| \leq \int_{-\infty}^{x^i} | \partial_i u(x'_i(s))| ~\D s \leq \int_{-\infty}^\infty \frac{1}{w_\tau\circ x'_i(s)}| L^i u(x'_i(s))|\ \D s.
\end{equation}
This implies
\begin{equation}
	|u(x)|^{\frac{d}{d-1}} \leq \prod_{i = 1}^d \left(\int_{\R} \frac{|L^i u(x'_i(s))|}{w_\tau\circ x'_i(s)}\ \D s \right)^{\frac{1}{d-1}}.
\end{equation}
Now, integrating the left hand side and applying H\"older's inequality (exactly as in \cite{Nirenberg59}) this implies (noting that the volume form is weighted according to \eqref{eq:volformweight})
\begin{equation}\label{eq:mainGNS}
	\tau^{\frac{1}{d-1}} \int_{\Sigma_\tau}  w_\tau(x) |u(x)|^{\frac{d}{d-1}} \ \dvol \leq \prod_{i = 1}^d \left( \int_{\Sigma_\tau} |L^i u(x)| \ \dvol \right)^{\frac{1}{d-1}}. \tag{GNS\textsubscript{1}}
\end{equation}
The extra factor of $\tau$ comes from the $\dvol$ that appears different number of times on the two sides. 
Taking advantage of \eqref{eq:weightder} which shows that we have really an exponential-type weight, \eqref{eq:mainGNS} implies the following arbitrarily-weighted counterpart. For any $\alpha\in \R$,  
\begin{equation}\label{eq:wtGNS}
	\tau^{\frac{1}{d-1}} \int_{\Sigma_\tau} w_\tau^{1 + \alpha\cdot \frac{d}{d-1}} |u(x)|^{\frac{d}{d-1}} \ \dvol \leq \prod_{i = 1}^d \left( \int_{\Sigma_\tau} w_\tau^\alpha |L^i u| + |\alpha| w_\tau^\alpha |u| \ \dvol \right)^{\frac{1}{d-1}}. \tag{GNAWS\textsubscript{1}}
\end{equation}
(This last inequality follows by replacing $u \mapsto w_\tau^\alpha u$ in \eqref{eq:mainGNS}.)

In view of the form of the inequalities, we will introduce the following notations for weighted Sobolev spaces on $\Sigma_\tau$:
\begin{itemize}
	\item For $p\in [1,\infty)$ and $\alpha\in\R$, by $\Lebw^p_{\alpha}$ we refer to the weighted Lebesgue norm
		\[ \|u\|_{\Lebw^p_\alpha} = \left( \int w_\tau^\alpha |u|^p \ \dvol \right)^{1/p}. \]
	\item For $p\in [1,\infty)$, $\alpha\in\R$, and $k\in \mathbb{N}$, by $\mr\Sobw^{k,p}_\alpha$ we refer to the weighted homogeneous Sobolev norm
		\[ \|u\|_{\mr\Sobw^{k,p}_\alpha} = \sum_{i_1, \ldots, i_k = 1}^d \| L^{i_1} \cdots L^{i_k} u\|_{\Lebw^p_\alpha}.\]
		The corresponding inhomogeneous version $\Sobw^{k,p}_\alpha$ is 
		\[ \|u\|_{\Sobw^{k,p}_\alpha} = \sum_{j = 0}^k \|u\|_{\mr\Sobw^{j,p}_\alpha}.\]
\end{itemize}
So \eqref{eq:wtGNS} asserts the continuous embedding $\Sobw^{1,1}_\alpha \hookrightarrow \Lebw^{d/(d-1)}_{\alpha d / (d-1) + 1}$. 

\begin{rmk}\label{rmk:energycontrol}
	To foreshadow our discussion, notice that the standard $t$-energy of the linear wave equation (see \cite{Wong2017}) controls 
	\[  \tau^{-1} \|u\|^2_{\mr\Sobw^{1,2}_{-1}} + \tau \|\partial_t u\|^2_{\Lebw^2_{-1}}.\]
	On the other hand, the $t$-energy of the linear Klein-Gordon equation controls
	\[  \tau^{-1} \|u\|^2_{\mr\Sobw^{1,2}_{-1}} + \tau \|\partial_t u\|^2_{\Lebw^2_{-1}} + \tau^{-1} \|u\|^2_{\Lebw^2_1} \]
	(note the different weight on the final term).
\end{rmk}

Replacing $u$ by $u^q$, coupled with an application of H\"older's inequality, gives the standard extensions of \eqref{eq:mainGNS} and \eqref{eq:wtGNS} to $\Sobw^{1,p}_\alpha$. 
Let $1 \leq p < d$, we have
\begin{gather}
	\tau^{1/d} \|u\|_{\Lebw^{dp/(d-p)}_1} \lesssim \|u\|_{\mr\Sobw^{1,p}_{1-p}}, \tag{GNS\textsubscript{p}} \label{eq:pGNS}\\
	\tau^{1/d} \|u\|_{\Lebw^{dp/(d-p)}_{1+ \alpha d p / (d-p)}} \lesssim \|u\|_{\Sobw^{1,p}_{1-p + \alpha p}}. \label{eq:wtpGNS} \tag{GNAWS\textsubscript{p}}
\end{gather}
Iterating \eqref{eq:wtpGNS} above, we also have as a corollary that, given $k\in \mathbb{N}$ and $p\in [1,\infty)$ such that $kp < d$, for any $\beta\in \R$, 
\begin{equation}\label{eq:wtkGNS}
	\tau^{k/d} \|u\|_{\Lebw^{q}_{1 - q + q(\beta+k)}} \lesssim \|u\|_{\Sobw^{k,p}_{1-p + p\beta}}, \tag{GNAWS\textsubscript{pk}}
\end{equation}
where $q = dp / (d - kp)$ is the usual Sobolev conjugate of $p$. 
We note that the case $\beta + k = 1$ is essentially a re-formulation of the standard Gagliardo-Nirenberg-Sobolev inequality on $\R^d$. 

\begin{rmk}\label{rmk:morrey}
	Notice that formally setting $p = 2$, $k = d/2$, and $\beta = 0$, one sees that \eqref{eq:wtkGNS} has the correct scaling for an inequality of the type
	\[ \tau^{1/2} \|w_\tau^{d/2 - 1} u\|_{L^\infty} \quad \text{``} \lesssim \text{''} \quad  \|u\|_{\Sobw^{d/2,2}_{-1}}. \]
 	This inequality, as we know, is not true, due to the failure of the end-point Sobolev inequality into $L^\infty$. 
	On the other hand, the (Morrey-type) global Sobolev inequality as stated and proved in \cite{Wong2017} can be restated in the following form
	\begin{equation}
	\tau^{1/2} \|w_\tau^{d/2 - 1}u \|_{L^\infty} \lesssim \|u\|_{\Sobw^{\lfloor d/2 \rfloor + 1, 2}_{-1}}. \label{Morrey}
	\end{equation}
\end{rmk}

\subsection{Interpolating inequalities: non-borderline case}

The inequalities \eqref{eq:pGNS} and \eqref{eq:wtpGNS} represent the endpoint Sobolev embeddings, when $p < d$, in our setting. 
In this section we prove Gagliardo-Nirenberg type interpolation inequalities. 
For simplicity we will focus on the case of 1 derivative: that is, we examine embeddings of the form 
\[ \Sobw^{1,p}_\alpha \cap \Lebw^q_\beta \hookrightarrow \Lebw^r_\gamma \]
with $q \leq r \leq dp/(d-p)$. The case of higher derivatives, based on \eqref{eq:wtkGNS}, is analogous and left to the reader. For convenience we denote $p^* \eqdef \tfrac{dp}{d - p}$ as the Sobolev conjugate of $p$.

\begin{prop}
	Given $q \leq r \leq p^*$, and let $\theta \in [0,1]$ satisfy
\[ \frac{1}{r} = \frac{\theta}{q} + \frac{1-\theta}{p^*}.\]
Then the following inequalities hold for any $\alpha,\beta \in \R$:
\begin{align}
\tau^{(1-\theta)/d} \| u\|_{\Lebw_{1 + \theta \beta r}^r} & \lesssim \left(\| u\|_{\Lebw_{1+\beta q}^q}\right)^{\theta } \cdot\left( \| u\|_{\mr \Sobw_{1-p}^{1,p}}\right)^{1-\theta} \tag{GNS{\textsubscript{pqr}}}  \label{eq:pqrGNS}, \\
\tau^{(1-\theta)/d} \| u\|_{\Lebw_{1 +  (\theta \beta + (1-\theta) \alpha)\cdot r}^r} & \lesssim\left( \| u\|_{\Lebw_{1+\beta q}^q}\right)^{\theta }\cdot \left( \|u\|_{\Sobw_{1 - p +\alpha p}^{1,p}}\right)^{1-\theta}. \tag{GNAWS\textsubscript{pqr}} \label{eq:wtpqrGNS}
\end{align}
\end{prop}

\begin{proof}
	The inequalities hold by applying the following elementary interpolation inequality of the weighted $\Lebw^p_\alpha$ spaces: for all $\theta \in [0,1]$,
	\begin{equation} \| u\|_{\Lebw_{\beta\theta + (1-\theta)\alpha}^r} \le \| u\|_{\Lebw_{\beta q/r}^q}^{\theta}\cdot \| u\|_{\Lebw_{\alpha p/r}^p}^{1-\theta}, \label{interpolate}
	\end{equation}
whenever
\[\frac{1}{r} = \frac{\theta}{q} + \frac{1-\theta}{p}.\] 
\end{proof}

\subsection{Interpolating inequalities: borderline case}

In the previous section we treated the interpolation inequalities when $p < d$. 
In this section we treat the interpolation inequalities when $p = d$. 
Specifically, we examine embeddings of the form
\[ \Sobw^{1,d}_\alpha \cap \Lebw^q_\beta \hookrightarrow \Lebw^r_\gamma \]
where now $1 \le q \leq r < \infty$. In view of our applications, the case $p = d = 2$ will be of specific interest. We occasionally abbreviate the Sobolev conjugate $1^* = d/(d-1)$.

\begin{prop}
	Let $q \leq r < \infty$, and $\beta\in \R$. Then 
	\begin{equation} \label{eq:pdrGNS} \tag{GNS\textsubscript{pdr}}
	\left(\tau^{1/d}\right)^{\frac{r-q}{r}} \| u\|_{\Lebw_{1 +\theta \beta  r}^r} \lesssim\left( \|u\|_{\Lebw_{1+\beta q}^q }\right)^{q/r} \cdot \left( \|u\|_{\mr \Sobw_{(1-d)(1 + \beta\theta r)}^{1,d}}\right)^{(r-q)/r},
\end{equation}
	where $\theta\in (0,1]$ is the solution to 
	\[ \frac{1}{r} = \frac{\theta}{q} + \frac{1-\theta}{r+1^*}.\]
\end{prop}
\begin{proof}
	Replacing $u \mapsto u^{1 + r/1^*}$ in \eqref{eq:mainGNS} implies
\begin{align*}
\tau^{1/d} \left( \int w_\tau |u|^{r + 1^*} \ \dvol \right)^{1/1^*} & \lesssim \sum_{i=1}^d \int |u|^{r/1^*} | L^i u| \ \dvol \\
								    & \lesssim \left( \int w_\tau^{1 + \theta\beta r} |u|^r \ \dvol\right)^{1/1^*} \cdot \| u\|_{\mr \Sobw_{(1-d)(1 + \beta\theta r)}^{1,d}}
\end{align*}
by H\"older's inequality. Here we used that 
\[1 = w_\tau^{(1+\theta\beta r)/1^*}\cdot w_\tau^{(-1-\theta\beta r)/1^*}. \] 
This in particular implies
\begin{equation} \label{interp}
	\tau^{1/(d-1)} \| u\|_{\Lebw_1^{r + 1^*}}^{r + 1^*} \lesssim \| u\|_{\Lebw_{1+\theta\beta r}^r}^{r} \| u\|_{\mr \Sobw_{(1-d)(1 + \theta\beta r)}^{1,d}}^{1^*}.
\end{equation}
We next interpolate using \eqref{interpolate} to find
\[\| u\|_{\Lebw_{1 + \theta \beta r}^r} \le \left( \|u\|_{\Lebw_{1+\beta q}^q}\right)^\theta \cdot \left( \|u\|_{\Lebw_1^{r + 1^*}}\right)^{1-\theta}.\]
Plugging \eqref{interp} in, cancelling the extra factors on both sides, we get the desired inequality after noting that $\theta$ is given by 
\[ \theta = \frac{1^* q}{r(1^* +r - q)},\qquad 1-\theta = \frac{(r-q)(r + 1^*)}{r(1^* + r - q)}.\]
\end{proof}

We note that when $\beta = 0$, the triple of weights
\[ (1 + \theta\beta r, 1 + \beta q, (1-d)(1 + \beta\theta r)) = (1,1,1-d).\]
Replacing $u\mapsto w_\tau^{\alpha} u$ we further have as a corollary 
\begin{equation}\label{eq:pdrGNAWS}
	\tag{GNAWS\textsubscript{pdr}}
	\left(\tau^{1/d}\right)^{\frac{r-q}{r}} \| u\|_{\Lebw_{1 +\theta \beta  r + \alpha r}^r} \lesssim\Bigl( \|u\|_{\Lebw_{1+\beta q + \alpha q}^q }\Bigr)^{q/r} \Bigl( \|u\|_{\Sobw_{(1-d)(1 + \beta\theta r) + \alpha d}^{1,d}}\Bigr)^{(r-q)/r}.
\end{equation}

\section{Linear estimates}
\label{sect:lin:est}

In this section we apply our results to obtain $\Lebw_*^r$ bounds by $\Lebw_*^2$ integrals that occur as part of the conserved energy for the linear wave and Klein-Gordon equations. 
As we will see there is often more than one way to obtain interpolated estimates, depending on the number of derivatives one is willing to sacrifice. 
Rather than attempt to be exhaustive in this section, we will opt for concreteness and list several possible estimates for dimensions $d = 2, 3, 4$, where the choices are more limited. 
Throughout we will let $u$ be a smooth function on $\R^{1+d}$, and $u_t$ will denote its time derivative. If $\alpha$ is an $m$-tuple with elements drawn from $\{1,\dots,d\}$ (namely that $\alpha = (\alpha_1,\dots,\alpha_m)$ with $\alpha_i \in \{1,\dots,d\}$) we denote
\[ L^\alpha u \eqdef L^{\alpha_m}L^{\alpha_{m-1}}\cdots L^{\alpha_1} u.\]
By $|\alpha|$ we refer to its length, namely $m$. 

\subsection{Wave equation, \texorpdfstring{$d = 3,4$}{d=3,4}}
When $d \geq 3$, we can make use of the Hardy inequality (see \cite{Wong2017}) and obtain that $\|u\|_{\Lebw_{-1}^2} \lesssim \|u\|_{\mr\Sobw_{-1}^{1,2}}$. Therefore we will denote by $\mathfrak{E}_k$ the $k$th order energy quantity
\[ \mathfrak{E}_k(\tau) = \tau^{-1/2} \|u\|_{\Sobw^{k+1,2}_{-1}(\Sigma_\tau)} + \tau^{1/2} \|u_t\|_{\Sobw^{k,2}_{-1}(\Sigma_\tau)}. \]
If $u$ solves the linear wave equation with initial data $u(0,x) = u_0(x)$ and $u_t(0,x) = u_1(x)$, then $\mathfrak{E}_k(\tau)$ is uniformly bounded by $\|u_0\|_{W^{k+1,2}} + \|u_1\|_{W^{k,2}}$. Here $W^{k,p}$ are the standard Sobolev spaces on $\R^d$. 

\begin{prop}[$d = 3$]
	When $r \in [2,6]$,  
	\begin{align}
		\tau^{-1/r} \| u\|_{\Lebw_{r/2 - 2}^r(\Sigma_\tau)} & \lesssim \mathfrak{E}_0(\tau), \label{pqrGNS-wave-d=3} \\
		\tau^{-1/r}\left( \|u\|_{\mr \Sobw_{r/2 - 2}^{k+1,r}(\Sigma_\tau)} + \tau \|u_t\|_{\mr \Sobw_{r/2-2}^{k,r}(\Sigma_\tau)}\right) & \lesssim \left( \mathfrak{E}_k(\tau)\right)^{\frac{6-r}{2r}} \cdot \left( \mathfrak{E}_{k+1}(\tau)\right)^{\frac{3r-6}{2r}} . \label{pqrGNS-waveder-d=3}
	\end{align}
	When $r > 6$, 	
	\begin{align}
		\tau^{-1/r} \|u\|_{\Lebw_{r/2 -2}^{r}(\Sigma_\tau)}&  \lesssim \left(\mathfrak{E}_0(\tau)\right)^{\frac{r+6}{2r}} \cdot \left( \mathfrak{E}_{1}(\tau)\right)^{\frac{r-6}{2r}}, \label{p3rGNS-wave-ver1} \\
			\tau^{-1/r}\left(  \|u\|_{\mr \Sobw_{r/2 - 2}^{k+1,r}(\Sigma_\tau)}  + \tau \|u_t\|_{\mr \Sobw_{r/2-2}^{k,r}(\Sigma_\tau)}\right) &  \lesssim \left(\mathfrak{E}_{k+1}(\tau)\right)^{\frac{r+6}{2r}} \cdot \left( \mathfrak{E}_{k+2}(\tau)\right)^{\frac{r-6}{2r}} . \label{p3rGNS-waveder-ver1}
	\end{align}
	For higher derivatives, the latter of the above estimate in $r > 6$ can be replaced by
	\begin{equation}
	\tau^{-1/r} \left( \| u\|_{\mr \Sobw_{r/2-2}^{k+1,r}(\Sigma_\tau)} + \tau \|u_t\|_{\mr \Sobw_{r/2-2}^{k,r}(\Sigma_\tau)} \right) \lesssim \left( \mathfrak{E}_k(\tau)\right)^{\frac{4}{2r}} \cdot \left( \mathfrak{E}_{k+1}(\tau)\right)^{\frac{r-2}{2r}}\cdot \left( \mathfrak{E}_{k+2}\right)^{\frac{r-2}{2r}} \label{p3rGNS-waveder-ver2}
	\end{equation}
\end{prop}

\begin{proof}	
Estimate \eqref{pqrGNS-wave-d=3} follows by applying \eqref{eq:pqrGNS} with $d = 3, \ q = 2$. Indeed, we see
\[ \tau^{(1-\theta)/3} \| u\|_{\Lebw_{1-\theta r}^r} \lesssim \left( \|u\|_{\Lebw_{-1}^2}\right)^\theta \cdot \left( \|u\|_{\mr \Sobw_{-1}^{1,2}}\right)^{1-\theta},\]
where $\theta \in [0,1]$ is the solution to 
\[\frac{1}{r} = \frac{\theta}{2} + \frac{1-\theta}{6}\qquad \Longrightarrow \qquad \theta = \frac{6-r}{2r}, \quad 1-\theta = \frac{3r-6}{2r}.\]
Rearranging using Hardy on the first factor and the definition of the energy we see that \eqref{pqrGNS-wave-d=3} follows. Similarly, if $\alpha$ is a $k$-tuple with elements drawn from $\{1,2,3\}$ and $v$ is any function we have
\[ \tau^{(1-\theta)/3} \|L^\alpha v\|_{\Lebw_{1-\theta r}^r} \lesssim \left( \|L^\alpha v\|_{\Lebw_{-1}^2}\right)^\theta \cdot \left( \|L^\alpha v\|_{\mr \Sobw_{-1}^{1,2}}\right)^{1-\theta},\]
with the same $\theta$ as before. Replacing $v \mapsto L^iu$ or $u_t$, and since we can estimate $\| L^\alpha L^i u\|_{\Lebw_{-1}^2}$ by the $k$th order energy without invoking Hardy, \eqref{pqrGNS-waveder-d=3} follows using the definition of their energies with the respective weights.

For larger $r$, we first appeal to \eqref{eq:pdrGNAWS} with $d = 3$, $q = 6$, and
\begin{gather*}
	1 + \beta q + \alpha q = 1\\
	-2 (1 + \beta \theta r) + 3\alpha = -\frac12\\
	\theta r = \frac{9}{\frac32 + r - 6} \\
\end{gather*}
which is solved by
\[ - \alpha = \beta = - \frac{\frac32 + r - 6}{3 + 2r} .\]
This implies
\[ \tau^{\frac{r-6}{3r}} \| u\|_{\Lebw_{r/2 - 2}^r} \lesssim \|u\|_{\Lebw_1^6}^{6/r}\cdot \|u\|_{\Sobw_{-1/2}^{1,3}}^{(r-6)/r}.\]
Applying \eqref{pqrGNS-wave-d=3} and \eqref{pqrGNS-waveder-d=3} to the two terms on the right we get
\[
	\tau^{\frac{r-6}{3r}} \| u\|_{\Lebw_{r/2 - 2}^r} \lesssim \bigl(\tau^{1/6} \mathfrak{E}_0(\tau)\bigr)^{6/r} \cdot 
	\bigl(\tau^{1/3} \mathfrak{E}_0(\tau)^{1/2} \mathfrak{E}_1(\tau)^{1/2}\bigr)^{(r-6)/r}
\]
and
\[
	\tau^{\frac{r-6}{3r}} \| u\|_{\mr\Sobw_{r/2 - 2}^{k,r}} \lesssim \bigl(\tau^{1/6} \mathfrak{E}_k(\tau)\bigr)^{6/r} \cdot 
	\bigl(\tau^{1/3} \mathfrak{E}_{k}(\tau)^{1/2} \mathfrak{E}_{k+1}(\tau)^{1/2}\bigr)^{(r-6)/r}.
\]
Rearranging this gives \eqref{p3rGNS-wave-ver1} and \eqref{p3rGNS-waveder-ver1}

To find the other estimate for $r > 6$ we appeal to the borderline \eqref{eq:pdrGNAWS} inequality slightly differently. 
Using $d = 3$ and $q = 2$ now, with
\begin{gather*}
		1 + \beta q + \alpha q = -1,\\
		(1-d)(1 + \theta\beta r) + \alpha d = -1/2,\\
		1/r = \theta/q + (1-\theta)/(r + 1^*),
\end{gather*}
we can solve to find
\[ \theta = \frac{6}{r(2r-1)},\qquad \beta = \frac{(-3)(2r-1)}{2(3+2r)},\qquad \alpha = \frac{2r-9}{2(3+2r)}. \] 
Let $\alpha$ be a $k$-tuple with elements drawn from $\{1,2,3\}$ and $v$ be any function. Then the inequality reads  
\[ \left(\tau^{1/3}\right)^{\frac{r-2}{r}} \| L^\alpha v\|_{\Lebw_{r/2-2}^r} \lesssim  \left(\| L^\alpha v\|_{\Lebw_{-1}^2}\right)^{2/r} \cdot \left( \| L^\alpha v\|_{\Sobw_{-1/2}^{1,3}}\right)^{\frac{r-2}{r}}.\]
Estimating the second factor using \eqref{pqrGNS-waveder-d=3} with $r = 3$ and the choice $v = L^i u$ or $u_t$, we can then rearrange to obtain \eqref{p3rGNS-waveder-ver2}.
\end{proof}

\begin{prop}[$d = 4$]
	When $r \in [2,4]$, 
	\begin{align}
		\tau^{-1/r} \|u\|_{\Lebw_{r-3}^r(\Sigma_\tau)} &\lesssim \mathfrak{E}_0(\tau), \label{pqrGNS-wave-d=4} \\
		\tau^{-1/r} \left( \| u\|_{\mr \Sobw_{r-3}^{k+1,r}} + \tau \| u_t\|_{\mr \Sobw_{r-3}^{k,r}(\Sigma_\tau)} \right) & \lesssim \left( \mathfrak{E}_k(\tau)\right)^{\frac{4-r}{r}} \left(\mathfrak{E}_{k+1}(\tau)\right)^{\frac{2r-4}{r}}.\label{pqrGNS-waveder-d=4} 
	\end{align}
	When $r > 4$, 
	\begin{align}
	\tau^{-1/r} \| u\|_{\Lebw_{r-3}^r(\Sigma_\tau)} & \lesssim \left( \mathfrak{E}_0(\tau)\right)^{2/r} \cdot \left( \mathfrak{E}_1(\tau)\right)^{\frac{r-2}{r}} \label{p4rGNS-wave-ver1} \\ 
		\tau^{-1/r} \left( \| u \|_{\mr\Sobw_{r-3}^{k+1,r}(\Sigma_\tau)} + \tau \|u_t\|_{\mr \Sobw_{r-3}^{k,r}(\Sigma_\tau)}\right) & \lesssim \left( \mathfrak{E}_k(\tau) \right)^{2/r} \left(  \mathfrak{E}_{k+2}(\tau)\right)^{\frac{r-2}{r}} \label{p4rGNS-waveder-ver1},
	\end{align}
	or
	\begin{align}
		\tau^{-1/r}\|u\|_{\Lebw_{r-3}^{r}(\Sigma_\tau)} & \lesssim \left( \mathfrak{E}_{0}(\tau)\right)^{4/r} \left( \mathfrak{E}_{1}(\tau)\right)^{\frac{r-4}{r}}\label{p4rGNS-wave-ver2}, \\
		\tau^{-1/r}\left( \| u\|_{\mr \Sobw_{r-3}^{k+1,r}(\Sigma_\tau)} + \tau \|u_t\|_{\mr\Sobw_{r-3}^{k,r}(\Sigma_\tau)} \right)& \lesssim \left( \mathfrak{E}_{k+1}(\tau)\right)^{4/r} \left( \mathfrak{E}_{k+2}(\tau)\right)^{\frac{r-4}{r}}.\label{p4rGNS-waveder-ver2}
	\end{align}
\end{prop}

\begin{proof}
The proofs of \eqref{pqrGNS-wave-d=4} and \eqref{pqrGNS-waveder-d=4} are the same as \eqref{pqrGNS-wave-d=3} and \eqref{pqrGNS-waveder-d=3} except that now $d = 4$ and $\theta$ solves 
\[ \frac{1}{r} = \frac{\theta}{2} + \frac{1-\theta}{4} \qquad \Longrightarrow \qquad \theta = \frac{4-r}{r}, \quad 1-\theta = \frac{2r - 4}{r}.\] 

To find estimate for $r > 4$ we appeal to the borderline \eqref{eq:pdrGNAWS} inequality. We will first be applying the inequality with
\begin{gather*}
		d = 4,\\
		q = 2,\\
		1 + \beta q + \alpha q = -1,\\
		(1-d)(1 + \theta\beta r) + \alpha d = 1,\\
		1/r = \theta/q + (1-\theta)/(r + 1^*).
\end{gather*}
These equations are solved by
\[ \theta = \frac{8}{r(3r-2)},\qquad \beta = \frac{(-2)(3r-2)}{4+3r},\qquad \alpha = \frac{3r - 8}{4 + 3r}, \] 
and so the weight $1 + \theta\beta r + \alpha r = r - 3$. 

Let $\alpha$ be a $k$-tuple with elements drawn from $\{1,2,3,4\}$ and $v$ be any function. Then
\[ \left(\tau^{1/4}\right)^{\frac{r-2}{r}} \| L^\alpha v\|_{\Lebw_{r-3}^r} \lesssim \left( \|L^\alpha v\|_{\Lebw_{-1}^2}\right)^{2/r} \cdot \left( \| L^\alpha v\|_{\Sobw^{1,4}_1}\right)^{\frac{r-2}{r}}.\]
This inequality holds for $r > 2$, so in particular for $r > 4$. If $k = 0$ and $v = u$, then the first factor can be estimated by the energy after invoking Hardy. The second factor can by treated with \eqref{eq:pGNS} because $2^* = 4$:
\[ \| u\|_{\Sobw_1^{1,4}} = \| u\|_{\Lebw_1^4} + \|u\|_{\mr \Sobw_1^{1,4}} \lesssim \tau^{-1/4} (\| u \|_{\mr \Sobw_{-1}^{1,2}} + \|u\|_{\mr \Sobw_{-1}^{2,2}}).\] This gives \eqref{p4rGNS-wave-ver1} after applying the definition of the energy. Again, note that if $k$ is arbitrary and $v = L^i u$, then we do \emph{not} have to invoke Hardy to estimate the first factor by the energy $\tau^{1/r}\mathfrak{E}_{k}^{2/r}$. On the other hand, if  $v =  u_t$,  the first factor is bounded by $\tau^{-1/r} \mathfrak{E}_k^{2/r}$. The second factor in the case of $v \mapsto (L^i u, u_t)$ can again be treated with \eqref{eq:pGNS}. Rearranging the inequalities and using the coercivity of their energies with the respective weights gives \eqref{p4rGNS-waveder-ver1}.

Alternatively, we can also solve with
\begin{gather*}
		d = 4\\
		q = 4\\
		1 + \beta q + \alpha q = 1\\
		(1-d)(1 + \theta\beta r) + \alpha d = 1\\
		1/r = \theta/q + (1-\theta)/(r + 1^*).
\end{gather*}
Let $\alpha$ be a $k$-tuple now and compute again with \eqref{eq:pdrGNAWS} and \eqref{eq:pGNS}
\begin{align*}
\left( \tau^{1/4}\right)^{\frac{r-4}{r}} \| L^\alpha v\|_{\Lebw_{1 + \theta \beta r + \alpha r}^r}&  \lesssim \left( \|L^\alpha v \|_{\Lebw_1^4}\right)^{4/r} \cdot \left( \|L^\alpha v \|_{\Sobw_1^{1,4}}\right)^{\frac{r-4}{r}}.
\end{align*}
The prior equations are solved by 
\[ \theta = \frac{16}{r(3r - 8)}, \qquad - \beta = \alpha = \frac{3r-8}{4+3r}\] 
and so the weight $1 + \theta \beta r + \alpha r = r - 3$. We control each factor with \eqref{eq:pGNS} in the two cases $k = 0, \ v = u$ and arbitrary $k$ and $v = (L^i u, u_t)$ as above. This finishes the proof of \eqref{p4rGNS-wave-ver2} and \eqref{p4rGNS-waveder-ver2}.
\end{proof}

\subsection{Wave equation, \texorpdfstring{$d = 2$}{d=2}}
When $d = 2$, Hardy's inequality is generally unavailable for the wave equation energy. So the $k$th order energy should only be
\[ \mathfrak{E}_k(\tau) = \tau^{-1/2} \sum_{j = 1}^{k+1}\|u\|_{\mr\Sobw^{j,2}_{-1}(\Sigma_\tau)} + \tau^{1/2} \|u_t\|_{\Sobw^{k,2}_{-1}(\Sigma_\tau)}. \]
So we cannot in general control $\|u\|_{\Lebw_*^{r}}$; but we can control the first derivatives of $u$ in $\Lebw^r_*$ with suitable weights. 

\begin{prop}
	When $r \in[2,\infty)$, 
	\begin{align}
		\tau^{ - 1/r} \left( \|u\|_{\mr\Sobw_{-1}^{k+1,r}(\Sigma_\tau)} +\tau \|u_t\|_{\mr\Sobw_{-1}^{k,r}(\Sigma_\tau)} \right) \lesssim \left( \mathfrak{E}_{k}(\tau) \right)^{2/r} \left( \mathfrak{E}_{k+1}(\tau)\right)^{\frac{r-2}{r}}, \label{p2rGNS-waveder}
	\end{align}
\end{prop}
\begin{proof}
We appeal to the borderline \eqref{eq:pdrGNAWS} inequality with
	\begin{gather*}
		d = 2,\\
		q = 2,\\
		1 + \beta q + \alpha q = -1, \\
		(1-d)(1 + \theta\beta r) + \alpha d = -1,\\
		1/r = \theta / 2 + (1-\theta) / (r + 2).
	\end{gather*}
These equations are solved by 
\[ \theta = \frac{4}{r^2},\qquad \beta = \frac{-r}{2 + r},\qquad \alpha = \frac{-2}{2+r},\]
and so the weight $1 + \theta \beta r + \alpha r = -1$. 
Let $\alpha$ be a $k$-tuple and let $v$ be an arbitrary function. Then we compute  
\begin{align*}
 \left(\tau^{1/2}\right)^{\frac{r-2}{r}}\| L^\alpha v\|_{\Lebw_{-1}^r}  \lesssim \left( \|L^\alpha v\|_{\Lebw_{-1}^2}\right)^{2/r} \cdot \left( \| L^\alpha v\|_{\Sobw_{-1}^{1,2}} \right)^{\frac{r-2}{r}}.
\end{align*}
Replacing $v \mapsto L^iu$ or $u_t$ and using the coercivity of their energies with the respective weights concludes the proof.
\end{proof}

\subsection{Klein-Gordon equation, \texorpdfstring{$d = 2, 3,4$}{d = 2,3,4}}
The Klein-Gordon energies control additionally a differently weighted $L^2$ term. Moreover, as we will see below, it is useful to distinguish between the energies of $u$ and $u_t$ (the latter of which also solves the Klein-Gordon equation). We write the $k$th order energy as
\[ \mathfrak{E}_k[v](\tau) = \tau^{-1/2} \| v \|_{\Sobw^{k+1,2}_{-1}(\Sigma_\tau)} + \tau^{1/2} \|v_t\|_{\Sobw^{k,2}_{-1}(\Sigma_\tau)} + \tau^{-1/2} \|v\|_{\Sobw^{k,2}_{1}(\Sigma_\tau)},\]
where $v$ can play the roll of $u$ or $u_t$. Here we've assumed that $\tau \geq 1$, so that $\|u\|_{\Lebw_{-1}^2} \leq \|u\|_{\Lebw_{1}^2}$.

\begin{prop}[$d = 2$] \label{KG-d=2}
	When $r > 2$, we have
	\begin{align}
	\tau^{-1/r} \| u\|_{\mr \Sobw_1^{k,r}(\Sigma_\tau)} & \lesssim \mathfrak{E}_{k}[u](\tau) \label{p2rGNS-KGder-ver1}, \\
\tau^{-1/r} \| u\|_{\mr \Sobw_{r-1}^{k,r}(\Sigma_\tau)} & \lesssim \left(\mathfrak{E}_{k}[u](\tau)\right)^{2/r}\cdot \left( \mathfrak{E}_{k+1}[u]\right)^{\frac{r-2}{r}} \label{p2rGNS-KGder-ver2},\\
\tau^{-1/r} \|u\|_{\mr \Sobw_{-1}^{k+1,r}(\Sigma_\tau)} & \lesssim \left( \mathfrak{E}_k[u](\tau)\right)^{2/r}\cdot \left( \mathfrak{E}_{k+1}[u]\right)^{\frac{r-2}{r}} \label{p2rGNS-KGder-ver3},\\
\tau^{-1/r} \| u\|_{\mr \Sobw_{r-3}^{k+1,r}(\Sigma_\tau)} & \lesssim \left( \mathfrak{E}_{k}[u](\tau)\right)^{2/r} \cdot \left( \mathfrak{E}_{k+2}[u](\tau)\right)^{\frac{r-2}{r}} \label{p2rGNS-KGder-ver4}.
\end{align}
 For the time derivatives the following estimates hold:
\begin{align}
 \tau^{1-3/r} \| u_t\|_{\mr \Sobw_1^{k,r}(\Sigma_\tau)} & \lesssim \left( \mathfrak{E}_k[u_t](\tau)\right)^{{2/r}}\cdot \left(\mathfrak{E}_{k+1}[u](\tau)\right)^{\frac{r-2}{r}},\label{p2rGNS-KG-ut-ver1} \\
  \tau^{-1/r}  \| u_t\|_{\mr \Sobw_{r-1}^{k,r}(\Sigma_\tau)} & \lesssim \left( \mathfrak{E}_k[u_t](\tau)\right)^{2/r} \cdot \left( \mathfrak{E}_{k+1}[u_t](\tau)\right)^{\frac{r-2}{r}}\label{p2rGNS-KG-ut-ver2} ,\\
 \tau^{1-1/r} \|u_t\|_{\mr \Sobw_{-1}^{k,r}(\Sigma_\tau)} & \lesssim \left(\mathfrak{E}_k[u](\tau)\right)^{2/r} \cdot \left( \mathfrak{E}_{k+1}[u](\tau)\right)^{\frac{r-2}{r}}, \label{p2rGNS-KG-ut-ver3}\\
 \tau^{1/r} \| u_t\|_{\mr \Sobw_{r-3}^{k,r}(\Sigma_\tau)} & \lesssim \left( \mathfrak{E}_k[u](\tau)\right)^{2/r} \cdot \left( \mathfrak{E}_{k+1}[u_t](\tau)\right)^{\frac{r-2}{r}}  \label{p2rGNS-KG-ut-ver4}.
 \end{align}
\end{prop}

\begin{proof}
Throughout this proof $\alpha$ will be a $k$-tuple and $v$ will be an arbitrary function. We solve \eqref{eq:pdrGNAWS} for
	\begin{gather*}
		d = 2\\
		q = 2\\
		1 + \beta q + \alpha q = \mu\\
		(1-d)(1 + \theta\beta r) + \alpha d = \nu\\
		1/r = \theta / 2 + (1-\theta) / (r+2),
	\end{gather*}
where $\mu,\nu$ can take the values $\pm 1$. Denoting the weight
\[\rho(\mu,\nu) \eqdef 1 + \theta\beta r + \alpha r,\]
the borderline inequality yields
\begin{equation}
\left(\tau^{1/2}\right)^{\frac{r-2}{r}}\| L^\alpha v\|_{\Lebw_{\rho(\mu,\nu)}^r} \lesssim \left( \| u\|_{\Lebw_\mu^2}\right)^{2/r}\cdot \left( \| u\|_{\Sobw_\nu^{1,2}}\right)^{\frac{r-2}{r}}. \label{KGder-d=2}
\end{equation}
One explicitly computes the weights as
\[ \rho(1,-1) = 1,\qquad \rho(1,1) = r-1, \qquad \rho(-1,-1) = -1, \qquad \rho(-1,1) = r-3.\]
Replacing $v \mapsto (u,u_t)$ in \eqref{KGder-d=2} and using the definition of the energies with their respective weights with $\mu = 1, \nu = -1$ proves \eqref{p2rGNS-KGder-ver1}, \eqref{p2rGNS-KG-ut-ver1}. When $\mu = \nu = 1$, this proves \eqref{p2rGNS-KGder-ver2}, and \eqref{p2rGNS-KG-ut-ver2}. On the other hand, replacing $v \mapsto (L^iu,u_t)$ in \eqref{KGder-d=2} and using the definition of the energies with their respective weights with $\mu = \nu = -1$ shows \eqref{p2rGNS-KGder-ver3}, \eqref{p2rGNS-KG-ut-ver3}. Finally, using $\mu = -1, \nu = 1$ proves \eqref{p2rGNS-KGder-ver4}, and \eqref{p2rGNS-KG-ut-ver4}.

\end{proof}

\begin{rmk}
We note that \eqref{p2rGNS-KGder-ver3} and \eqref{p2rGNS-KG-ut-ver3} are identical to the estimates \eqref{p2rGNS-waveder} derived for the wave equation. Indeed, the Klein-Gordon and wave $t$-energies \emph{both} control
\[ \tau^{-1/2} \sum_{j=1}^{k+1} \|u\|_{\mr\Sobw_{-1}^{j,2}} + \tau^{1/2} \|u_t\|_{\Sobw_{-1}^{1,2}}.\] The takeaway is that the mass term $\tau^{-1/2} \|u\|_{\Lebw_1^2}$ allows for estimates with different weights.
\end{rmk}
\begin{rmk}
One can summarize the proof of Proposition \ref{KG-d=2} by saying that its estimates correspond to the four endpoint cases of $\mu,\nu = \pm 1$ when applying \eqref{eq:pdrGNAWS}. Of course, various interpolations of these hold. One can interpolate, for example,  equation \eqref{p2rGNS-KGder-ver1} with \eqref{p2rGNS-KGder-ver3} to see, for any $\theta \in [0,1]$,
\[\tau^{-1/r} \| u\|_{\mr \Sobw_{1-2\theta}^{k+1,r}} \lesssim \left( \mathfrak{E}_k[u](\tau)\right)^{2\theta/r}\cdot \left( \mathfrak{E}_{k+1}[u](\tau)\right)^{1 - 2\theta/r}.\]
For the sake of brevity and clarity, we leave these straightforward computations to the reader.
\end{rmk}

\begin{prop}[$d= 3$] \label{KG-d=3}
When $r \in [2,6]$, 
\begin{align}
\tau^{-1/r}\| u\|_{\mr \Sobw_1^{k,r}(\Sigma_\tau)} & \lesssim \mathfrak{E}_k[u](\tau) \label{pqrGNS-KGder-d=3-ver1} ,\\
\tau^{-1/r}\| u\|_{\mr \Sobw_{3r/2 - 2}^{k,r}(\Sigma_\tau)} & \lesssim \left( \mathfrak{E}_k[u]\right)^{\frac{6-r}{2r}} \cdot \left( \mathfrak{E}_{k+1}[u]\right)^{\frac{3r-6}{2r}}\label{pqrGNS-KGder-d=3-ver2}, \\
\tau^{-1/r}\| u\|_{\mr \Sobw_{r/2-2}^{k+1,r}(\Sigma_\tau)} & \lesssim \left( \mathfrak{E}_k[u]\right)^{\frac{6-r}{2r}} \cdot \left( \mathfrak{E}_{k+1}[u]\right)^{\frac{3r-6}{2r}} \label{pqrGNS-KGder-d=3-ver3}, \\
\tau^{-1/r}\| u\|_{\mr \Sobw_{2r-5}^{k+1,r}(\Sigma_\tau)} & \lesssim \left( \mathfrak{E}_k[u]\right)^{\frac{6-r}{2r}} \cdot \left( \mathfrak{E}_{k+2}[u]\right)^{\frac{3r-6}{2r}}\label{pqrGNS-KGder-d=3-ver4}.
\end{align}
For the time derivatives, the following estimates hold:
\begin{align}
 \tau^{3/2-4/r} \| u_t\|_{\mr \Sobw_1^{k,r}(\Sigma_\tau)} & \lesssim \left( \mathfrak{E}_k[u_t](\tau)\right)^{\frac{6-r}{2r}}\cdot \left(\mathfrak{E}_{k+1}[u](\tau)\right)^{\frac{3r-6}{2r}},\label{pqrGNS-KG-ut-d=3-ver1} \\
  \tau^{-1/r}  \| u_t\|_{\mr \Sobw_{3r/2-2}^{k,r}(\Sigma_\tau)} & \lesssim \left( \mathfrak{E}_k[u_t](\tau)\right)^{\frac{6-r}{2r}} \cdot \left( \mathfrak{E}_{k+1}[u_t](\tau)\right)^{\frac{3r-6}{2r}}\label{pqrGNS-KG-ut-d=3-ver2} ,\\
 \tau^{1-1/r} \|u_t\|_{\mr \Sobw_{r/2-2}^{k,r}(\Sigma_\tau)} & \lesssim \left(\mathfrak{E}_k[u](\tau)\right)^{\frac{6-r}{2r}} \cdot \left( \mathfrak{E}_{k+1}[u](\tau)\right)^{\frac{3r-6}{2r}}, \label{pqrGNS-KG-ut-d=3-ver3}\\
 \tau^{-1/2 + 2/r} \| u_t\|_{\mr \Sobw_{2r-5}^{k,r}(\Sigma_\tau)} & \lesssim \left( \mathfrak{E}_k[u](\tau)\right)^{\frac{6-r}{2r}} \cdot \left( \mathfrak{E}_{k+1}[u_t](\tau)\right)^{\frac{3r-6}{2r}}  \label{pqrGNS-KG-ut-d=3-ver4}.
 \end{align}
 When $r > 6$, the following estimates hold:
 \begin{align}
 \tau^{-1/r} \| u\|_{\mr \Sobw_{r-1}^{k,2}(\Sigma_\tau)} & \lesssim \left( \mathfrak{E}_k[u] (\tau)\right)^{2/r} \cdot \left( \mathfrak{E}_{k+1}[u]\right)^{\frac{r-2}{r}}, \label{p3rGNS-KGder-ver1}  \\
 \tau^{-1/r} \| u\|_{\mr \Sobw_{3r/2 - 2}^{k,2}(\Sigma_\tau)} & \lesssim  \left( \mathfrak{E}_k[u] (\tau)\right)^{2/r}\cdot \left( \mathfrak{E}_{k+1}[u] (\tau)^{1/2} \cdot  \mathfrak{E}_{k+2}[u](\tau)^{1/2}\right)^{\frac{r-2}{r}},\label{p3rGNS-KGder-ver2} 
\\
 \tau^{-1/r} \| u\|_{\mr \Sobw_{r/2}^{k,2}(\Sigma_\tau)} & \lesssim \left( \mathfrak{E}_k[u] (\tau)\right)^{\frac{r+2}{2r}} \cdot \left( \mathfrak{E}_{k+1}[u](\tau)\right)^{\frac{r-2}{2r}}, \label{p3rGNS-KGder-ver3} 
\\
  \tau^{-1/r} \| u\|_{\mr \Sobw_{r-1}^{k,2}(\Sigma_\tau)} & \lesssim \left( \mathfrak{E}_k[u] (\tau)\right)^{\frac{r+2}{2r}} \cdot \left( \mathfrak{E}_{k+2}[u]\right)^{\frac{r-2}{2r}}, \label{p3rGNS-KGder-ver4} 
\\
  \tau^{-1/r} \| u\|_{\mr \Sobw_{r-3}^{k+1,2}(\Sigma_\tau)} & \lesssim  \left( \mathfrak{E}_k[u] (\tau)\right)^{2/r} \cdot\left( \mathfrak{E}_{k+2}[u] (\tau)\right)^{\frac{r-2}{r}}, \label{p3rGNS-KGder-ver5} 
\\
   \tau^{-1/r} \| u\|_{\mr \Sobw_{3r/2 - 4}^{k+1,2}(\Sigma_\tau)} & \lesssim  \left( \mathfrak{E}_k[u] (\tau)\right)^{2/r} \cdot\left( \mathfrak{E}_{k+2}[u] (\tau)^{1/2} \cdot  \mathfrak{E}_{k+3}[u](\tau)^{1/2}\right)^{\frac{r-2}{r}},\label{p3rGNS-KGder-ver6} 
\\
    \tau^{-1/r} \| u\|_{\mr \Sobw_{r/2-2}^{k+1,2}(\Sigma_\tau)} & \lesssim  \left( \mathfrak{E}_k[u] (\tau)\right)^{2/r} \cdot\left( \mathfrak{E}_{k+1}[u] (\tau)^{1/2} \cdot  \mathfrak{E}_{k+2}[u](\tau)^{1/2}\right)^{\frac{r-2}{r}}, \label{p3rGNS-KGder-ver7} 
\\
  \tau^{-1/r} \| u\|_{\mr \Sobw_{r-3}^{k+1,2}(\Sigma_\tau)} & \lesssim  \left( \mathfrak{E}_k[u] (\tau)\right)^{2/r} \cdot\left( \mathfrak{E}_{k+1}[u] (\tau)^{1/2} \cdot  \mathfrak{E}_{k+3}[u](\tau)^{1/2}\right)^{\frac{r-2}{r}}.\label{p3rGNS-KGder-ver8} 
  \end{align}
For the time derivatives, we have:
\begin{align}
 \tau^{1/2 - 2/r} \| u\|_{\mr \Sobw_{r-1}^{k,2}(\Sigma_\tau)} & \lesssim \left( \mathfrak{E}_k[u_t] (\tau)\right)^{2/r} \cdot\left( \mathfrak{E}_{k+1}[u_t] (\tau)^{1/2} \cdot  \mathfrak{E}_{k+2}[u](\tau)^{1/2}\right)^{\frac{r-2}{r}}, \label{p3rGNS-KG-ut-ver1}  \\
 \tau^{-1/r} \| u\|_{\mr \Sobw_{3r/2 - 2}^{k,2}(\Sigma_\tau)} & \lesssim  \left( \mathfrak{E}_k[u_t] (\tau)\right)^{2/r} \cdot\left( \mathfrak{E}_{k+1}[u_t] (\tau)^{1/2} \cdot  \mathfrak{E}_{k+2}[u_t](\tau)^{1/2}\right)^{\frac{r-2}{r}},\label{p3rGNS-KG-ut-ver2} 
\\
 \tau^{1-3/r} \| u_t\|_{\mr \Sobw_{r/2}^{k,2}(\Sigma_\tau)} & \lesssim\left( \mathfrak{E}_k[u_t] (\tau)\right)^{2/r} \cdot\left( \mathfrak{E}_{k+1}[u] (\tau)^{1/2} \cdot  \mathfrak{E}_{k+2}[u](\tau)^{1/2}\right)^{\frac{r-2}{r}}, \label{p3rGNS-KG-ut-ver3} 
\\
  \tau^{1/2 - 2/r} \| u_t\|_{\mr \Sobw_{r-1}^{k,2}(\Sigma_\tau)} & \lesssim \left( \mathfrak{E}_k[u_t] (\tau)\right)^{2/r} \cdot\left( \mathfrak{E}_{k+1}[u] (\tau)^{1/2} \cdot  \mathfrak{E}_{k+2}[u_t](\tau)^{1/2}\right)^{\frac{r-2}{r}}, \label{p3rGNS-KG-ut-ver4} 
\\
  \tau^{1/2} \| u_t\|_{\mr \Sobw_{r-3}^{k,2}(\Sigma_\tau)} & \lesssim \left( \mathfrak{E}_k[u] (\tau)\right)^{2/r} \cdot\left( \mathfrak{E}_{k+1}[u_t] (\tau)^{1/2} \cdot  \mathfrak{E}_{k+2}[u](\tau)^{1/2}\right)^{\frac{r-2}{r}}, \label{p3rGNS-KG-ut-ver5} 
\\
   \tau^{1/r} \| u_t\|_{\mr \Sobw_{3r/2 - 4}^{k,2}(\Sigma_\tau)} & \lesssim  \left( \mathfrak{E}_k[u] (\tau)\right)^{2/r} \cdot\left( \mathfrak{E}_{k+1}[u_t] (\tau)^{1/2} \cdot  \mathfrak{E}_{k+2}[u_t](\tau)^{1/2}\right)^{\frac{r-2}{r}},\label{p3rGNS-KG-ut-ver6} 
\\
    \tau^{1-1/r} \| u_t\|_{\mr \Sobw_{r/2-2}^{k,2}(\Sigma_\tau)} & \lesssim  \left( \mathfrak{E}_k[u] (\tau)\right)^{2/r} \cdot \left( \mathfrak{E}_{k+1}[u] (\tau)^{1/2} \cdot  \mathfrak{E}_{k+2}[u](\tau)^{1/2}\right)^{\frac{r-2}{r}}, \label{p3rGNS-KG-ut-ver7} 
\\
  \tau^{1/2} \| u_t\|_{\mr \Sobw_{r-3}^{k,2}(\Sigma_\tau)} & \lesssim  \left( \mathfrak{E}_k[u] (\tau)\right)^{2/r} \cdot \left( \mathfrak{E}_{k+1}[u_t] (\tau)^{1/2} \cdot  \mathfrak{E}_{k+2}[u](\tau)^{1/2}\right)^{\frac{r-2}{r}}.\label{p3rGNS-KG-ut-ver8} 
  \end{align}

\end{prop}
\begin{proof}
Throughout this proof $\alpha$ will be a $k$-tuple and $v$ will be an arbitrary function. For $r \in [2,6]$, we can solve \eqref{eq:wtpqrGNS} with
	\begin{gather*}
		d = 3\\
		q = 2\\ 
		p = 2\\
		1 + \beta q = \mu\\
		1-p + \alpha p =  \nu\\
		1/r = \theta /q + (1-\theta) /p^*,
	\end{gather*}
where $\mu,\nu$ can again take the values $\pm 1$. Denoting the weight 
\[\rho_{pqr}(\mu,\nu) \eqdef 1 + (\theta \beta + (1-\theta)\cdot \alpha)r\]
the interpolation inequality yields
\begin{equation}
\tau^{1/2 - 1/r} \| L^\alpha v\|_{\Lebw_{\rho_{pqr}(\mu,\nu)}^r} \lesssim \left( \| L^ \alpha v\|_{\Lebw_\mu^2}\right)^{\frac{6-r}{2r}}\cdot \left( \| L^\alpha v\|_{\Sobw^{1,2}_\nu}\right)^{\frac{3r-6}{2r}}. \label{KGder-d=3}
\end{equation}
One explicitly computes the weights as 
\[ \rho_{pqr}(1,-1) = 1, \quad \rho_{pqr}(1,1) = 3r/2 - 2, \quad \rho_{pqr}(-1,-1) = r/2 - 2, \quad \rho_{pqr}(-1,1) = 2r - 5.\]
We note that we are unable to simply replace $v \mapsto u$ in \eqref{KGder-d=3} and use the definition of the energies with their respective weights with $\mu = 1, \nu = -1$ because the second factor in \eqref{KGder-d=3} is the \emph{inhomogeneous} Sobolev norm. To remedy this, we again use the extra mass term in the energy:  
\begin{align*}
 \| L^\alpha u \|_{\Sobw_{-1}^{1,2}} & = \|L^\alpha u\|_{\Lebw_{-1}^2} + \| L^\alpha u \|_{\mr \Sobw_{-1}^{1,2}} \\
 & \le \| L^\alpha u\|_{\Lebw_1^2} +  \| L^\alpha u \|_{\mr \Sobw_{-1}^{1,2}} \\
 & \lesssim \tau^{1/2} \mathfrak{E}_k[u].
\end{align*}
Now we can replace $v \mapsto(u,u_t)$ in \eqref{KGder-d=3} to prove \eqref{pqrGNS-KGder-d=3-ver1}, \eqref{pqrGNS-KG-ut-d=3-ver1} (note that this problem did not occur for $v = u_t$). When $\mu = \nu = 1$, $v \mapsto (u,u_t)$ in \eqref{KGder-d=3} also proves  \eqref{pqrGNS-KGder-d=3-ver2}, and \eqref{pqrGNS-KG-ut-d=3-ver2}.

On the other hand, replacing $v \mapsto (L^iu,u_t)$ in \eqref{KGder-d=3} and using the definition of the energies with their respective weights with $\mu = \nu = -1 $ shows \eqref{pqrGNS-KGder-d=3-ver3}, \eqref{pqrGNS-KG-ut-d=3-ver3}. Finally, using $\mu = -1, \nu = 1$ proves \eqref{pqrGNS-KGder-d=3-ver4}, and \eqref{pqrGNS-KG-ut-d=3-ver4}.

For the estimates when $r > 6$, we appeal to the borderline \eqref{eq:pdrGNAWS} inequality with 
\begin{gather*}
		d = 3,\\
		q = 2,\\
		1 + \beta q + \alpha q = \sigma,\\
		(1-d)(1 + \theta\beta r) + \alpha d = \rho_{pq3}(\mu,\nu),\\
		1/r = \theta/q + (1-\theta)/(r + 1^*),
\end{gather*}
where $\sigma$ can take the values $\pm 1$ and $\rho_{pq3}(\mu,\nu)$ is as above. This inequality is valid for $r > 2$ so in particular $r > 6$. Denoting the weight 
\[\rho_{d=3}(\sigma,\mu,\nu) \eqdef 1 + \theta\beta r + \alpha r,\] the borderline inequality yields
\begin{equation}
\left(\tau^{1/3}\right)^{\frac{r-2}{r}}\|L^\alpha v\|_{\Lebw_{\rho_{d=3}(\sigma,\mu,\nu)}^r}\lesssim \left( \|L^\alpha v\|_{\Lebw_\sigma^2}\right)^{2/r} \cdot \left( \| L^\alpha v\|_{\Sobw_{\rho_{pq3}(\mu,\nu)}^{1,3}}\right)^{\frac{r-2}{r}}.\label{KGder-d=3-border}
\end{equation}
One explicitly computes the weights 
\begin{align*}
&\rho_{d=3}(1,1,-1)  = r-1, && \rho_{d=3}(-1,1,-1)  = r-3, \\
&\rho_{d=3}(1,1,1)  = 3r/2-2, && \rho_{d=3}(-1,1,1) = 3r/2 - 4, \\
&\rho_{d=3}(1,-1,-1)  = r/2, & &\rho_{d=3}(-1,-1,-1)  = r/2-2, \\
&\rho_{d=3}(1,-1,1)  = r-1, & &\rho_{d=3}(-1,-1,1) = r-3.
\end{align*}
Note that even though 
\begin{align*}
\rho_{d=3}(1,1,-1) & = \rho_{d=3}(1,-1,1), \\
\rho_{d=3}(-1,1,-1) & = \rho_{d=3}(-1,-1,1),
\end{align*} that $\mu,\nu$ are different implies that we have different estimates. We note that replacing $v \mapsto u,\ L^iu$ whenever $\sigma = 1, -1$ (respectively) is not enough to prove the estimates because the second factor in \eqref{KGder-d=3-border} is
\[ \|L^\alpha v\|_{\Sobw_{\rho_{pq3}(\mu,\nu)}^{1,3}} = \| L^\alpha v\|_{\Lebw_{\rho_{pq3}(\mu,\nu)}^3} +  \|L^\alpha v\|_{\mr \Sobw_{\rho_{pq3}(\mu,\nu)}^{1,3}}.\]
Consequently, special care must be taken to analyze the two \emph{different} derivative terms because the left hand sides in \eqref{pqrGNS-KGder-d=3-ver1} - \eqref{pqrGNS-KGder-d=3-ver4} are all with respect to the \emph{homogeneous} spaces $\mr \Sobw_*^{*,r}$. 

Fix $v = u$. When $\mu = 1, - 1$ we can estimate 
\[ \tau^{-1/3} \|L^\alpha u\|_{\Sobw_1^{1,3}} \lesssim \mathfrak{E}_{k+1}[u]\]
by using $\mathfrak{E}_k \le \mathfrak{E}_{k+1}$.  Arguing in the same way, when $\mu = \nu = 1$ one finds
\[ \tau^{-1/3} \| L^\alpha u \|_{\Sobw_{5/2}^{1,3}} \lesssim \mathfrak{E}_{k+1}[u]^{1/2} \cdot \mathfrak{E}_{k+2}[u]^{1/2}.\] 
For $\mu = \nu = -1$, on the other hand, we estimate
\[ \tau^{-1/3} \| L^\alpha u \|_{\Sobw_{-1/2}^{1,3}} \le \tau^{-1/3}\left( \| L^\alpha u\|_{\Lebw_{5/2}^{3}} + \| L^\alpha u \|_{\mr \Sobw_{-1/2}^{1,2}}\right)\lesssim \mathfrak{E}_{k}[u]^{1/2}\cdot \mathfrak{E}_{k+1}[u]^{1/2}.\]
The first term was controlled again using \eqref{pqrGNS-KGder-d=3-ver2}. Finally, when $\mu = -1, \nu = -1$ we see 
\begin{align*}
\tau^{-1/3} \| L^\alpha u \|_{\Sobw_{1}^{1,3}} & =\tau^{-1/3} \left( \|L^\alpha u\|_{\Lebw_1^3} + \| L^\alpha u\|_{\mr \Sobw_{1}^{1,3}}\right)\\
&  \lesssim \mathfrak{E}_{k}[u]^{1/2}\cdot  \mathfrak{E}_{k+1}[u]^{1/2} +  \mathfrak{E}_{k}[u]^{1/2}\cdot  \mathfrak{E}_{k+2}[u]^{1/2} \\
&\lesssim \mathfrak{E}_{k}[u]^{1/2}\cdot  \mathfrak{E}_{k+2}[u]^{1/2}.
\end{align*}
Using these estimates in \eqref{KGder-d=3-border} with $\sigma = 1$ proves \eqref{p3rGNS-KGder-ver1} - \eqref{p3rGNS-KGder-ver4} after appealing to the definition of the energy with the respective weights. 

Fix now $v = L^i u$. Then, arguing as above with $\mathfrak{E}_{k} \le \mathfrak{E}_{k+ 1}$ for arbitrary $k$ to control $\|L^\alpha L^i u\|_{\Sobw_{\rho_{pq3}(\mu,\nu)}^{1,3}}$, equation \eqref{KGder-d=3-border} with $\sigma = -1$ and the estimates \eqref{pqrGNS-KGder-d=3-ver1} - \eqref{pqrGNS-KGder-d=3-ver4} with the respective choices of $\mu,\nu = \pm 1$ prove \eqref{p3rGNS-KGder-ver5} - \eqref{p3rGNS-KGder-ver8}.

The time derivative estimates are more straight forward, the $\sigma = \pm 1$ cases are treated separately but similarly. The first factor in \eqref{KGder-d=3-border} is treated by 
\[\|L^\alpha u_t\|_{\Lebw_1^2} \le \tau^{1/2} \mathfrak{E}_k[u_t],\qquad \| L^\alpha u_t\|_{\Lebw_{-1}^2} \le \tau^{-1/2} \mathfrak{E}_k[u].\] 
Simply replacing $v \mapsto u_t$ in \eqref{KGder-d=3-border} and using \eqref{pqrGNS-KG-ut-d=3-ver1} - \eqref{pqrGNS-KG-ut-d=3-ver4} to control the second factor $\|L^\alpha u_t\|_{\Sobw_{\rho_{pq3}(\mu,\nu)}^{1,3}}$ with the respective choices of $\mu,\nu = \pm1$ proves \eqref{p3rGNS-KG-ut-ver1} - \eqref{p3rGNS-KG-ut-ver8} after appealing to the energies with the respective weights. 
\end{proof}
\begin{rmk}
	For the estimates when $r > 6$ in the previous proof we made the choice of interpolating $\Lebw^2_{*}$ with $\Sobw^{1,3}_{*}$, see \eqref{KGder-d=3-border}.
	As we saw previously in the wave case, specifically the proof of \eqref{p3rGNS-waveder-ver1}, we can also obtain estimates interpolating $\Lebw^6_*$ with $\Sobw_*^{1,3}$ instead. 
	For brevity we leave out these cases and various other interpolations.
\end{rmk}

\begin{prop}[$d= 4$] \label{KG-d=4}
When $r \in [2,4]$, 
\begin{align}
\tau^{-1/r}\| u\|_{\mr \Sobw_1^{k,r}(\Sigma_\tau)} & \lesssim \mathfrak{E}_k[u](\tau) \label{pqrGNS-KGder-d=4-ver1} ,\\
\tau^{-1/r}\| u\|_{\mr \Sobw_{2r-3}^{k,r}(\Sigma_\tau)} & \lesssim \left( \mathfrak{E}_k[u]\right)^{\frac{4-r}{r}} \cdot \left( \mathfrak{E}_{k+1}[u]\right)^{\frac{2r-4}{r}}\label{pqrGNS-KGder-d=4-ver2}, \\
\tau^{-1/r}\| u\|_{\mr \Sobw_{r-3}^{k+1,r}(\Sigma_\tau)} & \lesssim \left( \mathfrak{E}_k[u]\right)^{\frac{4-r}{r}} \cdot \left( \mathfrak{E}_{k+1}[u]\right)^{\frac{2r-4}{r}} \label{pqrGNS-KGder-d=4-ver3}, \\
\tau^{-1/r}\| u\|_{\mr \Sobw_{3r-7}^{k+1,r}(\Sigma_\tau)} & \lesssim \left( \mathfrak{E}_k[u]\right)^{\frac{4-r}{r}} \cdot \left( \mathfrak{E}_{k+2}[u]\right)^{\frac{2r-4}{r}}\label{pqrGNS-KGder-d=4-ver4}.
\end{align}
For the time derivatives, the following estimates hold:
\begin{align}
 \tau^{2 - 5/r} \| u_t\|_{\mr \Sobw_1^{k,r}(\Sigma_\tau)} & \lesssim \left( \mathfrak{E}_k[u_t](\tau)\right)^{\frac{4-r}{r}}\cdot \left(\mathfrak{E}_{k+1}[u](\tau)\right)^{\frac{2r-4}{r}},\label{pqrGNS-KG-ut-d=4-ver1} \\
  \tau^{-1/r}  \| u_t\|_{\mr \Sobw_{2r-3}^{k,r}(\Sigma_\tau)} & \lesssim \left( \mathfrak{E}_k[u_t](\tau)\right)^{\frac{4-r}{r}} \cdot \left( \mathfrak{E}_{k+1}[u_t](\tau)\right)^{\frac{2r-4}{r}}\label{pqrGNS-KG-ut-d=4-ver2} ,\\
 \tau^{1-1/r} \|u_t\|_{\mr \Sobw_{r-3}^{k,r}(\Sigma_\tau)} & \lesssim \left(\mathfrak{E}_k[u](\tau)\right)^{\frac{4-r}{r}} \cdot \left( \mathfrak{E}_{k+1}[u](\tau)\right)^{\frac{2r-4}{r}}, \label{pqrGNS-KG-ut-d=4-ver3}\\
 \tau^{-1 + 3/r} \| u_t\|_{\mr \Sobw_{3r-7}^{k,r}(\Sigma_\tau)} & \lesssim \left( \mathfrak{E}_k[u](\tau)\right)^{\frac{4-r}{r}} \cdot \left( \mathfrak{E}_{k+1}[u_t](\tau)\right)^{\frac{2r-4}{r}}  \label{pqrGNS-KG-ut-d=4-ver4}.
\end{align}
When $r > 4$, the following estimates hold:
 \begin{align}
 \tau^{-1/r} \| u\|_{\mr \Sobw_{r-1}^{k,2}(\Sigma_\tau)} & \lesssim \begin{cases} \left( \mathfrak{E}_k[u] (\tau)\right)^{2/r} \cdot \left( \mathfrak{E}_{k+1}[u]\right)^{\frac{r-2}{r}},  \\
\left( \mathfrak{E}_k[u] (\tau)\right)^{\frac{r+2}{2r}} \cdot \left( \mathfrak{E}_{k+1}[u](\tau)\right)^{\frac{r-2}{2r}},
 \end{cases} \label{p4rGNS-KGder-ver1}  \\
 \tau^{-1/r} \| u\|_{\mr \Sobw_{2r-3}^{k,2}(\Sigma_\tau)} & \lesssim \begin{cases}
 \left( \mathfrak{E}_k[u] (\tau)\right)^{2/r}\cdot \left( \mathfrak{E}_{k+1}[u] (\tau)^{1/2} \cdot  \mathfrak{E}_{k+2}[u](\tau)^{1/2}\right)^{\frac{r-2}{r}},\\
 \left( \mathfrak{E}_k[u] (\tau)\right)^{\frac{r+2}{2r}} \cdot \left( \mathfrak{E}_{k+2}[u]\right)^{\frac{r-2}{2r}},
 \end{cases} \label{p4rGNS-KGder-ver2} \\
  \tau^{-1/r} \| u\|_{\mr \Sobw_{r-3}^{k+1,2}(\Sigma_\tau)} & \lesssim  \begin{cases}
  \left( \mathfrak{E}_k[u] (\tau)\right)^{2/r} \cdot\left( \mathfrak{E}_{k+2}[u] (\tau)\right)^{\frac{r-2}{r}}, \\
   \left( \mathfrak{E}_k[u] (\tau)\right)^{2/r} \cdot\left( \mathfrak{E}_{k+1}[u] (\tau)^{1/2} \cdot  \mathfrak{E}_{k+2}[u](\tau)^{1/2}\right)^{\frac{r-2}{r}}, 
   \end{cases} \label{p4rGNS-KGder-ver3} \\
   \tau^{-1/r} \| u\|_{\mr \Sobw_{2r-5}^{k+1,2}(\Sigma_\tau)} & \lesssim \begin{cases}
   \left( \mathfrak{E}_k[u] (\tau)\right)^{2/r} \cdot\left( \mathfrak{E}_{k+2}[u] (\tau)^{1/2} \cdot  \mathfrak{E}_{k+3}[u](\tau)^{1/2}\right)^{\frac{r-2}{r}}, \\
    \left( \mathfrak{E}_k[u] (\tau)\right)^{2/r} \cdot\left( \mathfrak{E}_{k+1}[u] (\tau)^{1/2} \cdot  \mathfrak{E}_{k+3}[u](\tau)^{1/2}\right)^{\frac{r-2}{r}}.
    \end{cases} \label{p4rGNS-KGder-ver4} 
  \end{align}
For the time derivatives, we have:
\begin{align}
 \tau^{1- 3/r} \| u_t\|_{\mr \Sobw_{r-1}^{k,2}(\Sigma_\tau)} & \lesssim \begin{cases}
 \left( \mathfrak{E}_k[u_t] (\tau)\right)^{2/r} \cdot\left( \mathfrak{E}_{k+1}[u_t] (\tau)^{1/2} \cdot  \mathfrak{E}_{k+2}[u](\tau)^{1/2}\right)^{\frac{r-2}{r}},  \\
 \left( \mathfrak{E}_k[u_t] (\tau)\right)^{2/r} \cdot\left( \mathfrak{E}_{k+1}[u] (\tau)^{1/2} \cdot  \mathfrak{E}_{k+2}[u](\tau)^{1/2}\right)^{\frac{r-2}{r}},
 \end{cases}  \label{p4rGNS-KG-ut-ver1} \\
 \tau^{-1/r} \| u_t\|_{\mr \Sobw_{2r-3}^{k,2}(\Sigma_\tau)} & \lesssim \begin{cases}
\left( \mathfrak{E}_k[u_t] (\tau)\right)^{2/r} \cdot\left( \mathfrak{E}_{k+1}[u_t] (\tau)^{1/2} \cdot  \mathfrak{E}_{k+2}[u_t](\tau)^{1/2}\right)^{\frac{r-2}{r}}, \\
 \left( \mathfrak{E}_k[u_t] (\tau)\right)^{2/r} \cdot\left( \mathfrak{E}_{k+1}[u] (\tau)^{1/2} \cdot  \mathfrak{E}_{k+2}[u_t](\tau)^{1/2}\right)^{\frac{r-2}{r}}, 
 \end{cases}  \label{p4rGNS-KG-ut-ver2}
\\
 \tau^{1-1/r} \| u_t\|_{\mr \Sobw_{r-3}^{k,2}(\Sigma_\tau)} & \lesssim \begin{cases}
 \left( \mathfrak{E}_k[u] (\tau)\right)^{2/r} \cdot\left( \mathfrak{E}_{k+1}[u_t] (\tau)^{1/2} \cdot  \mathfrak{E}_{k+2}[u](\tau)^{1/2}\right)^{\frac{r-2}{r}}, \\
  \left( \mathfrak{E}_k[u] (\tau)\right)^{2/r} \cdot \left( \mathfrak{E}_{k+1}[u] (\tau)^{1/2} \cdot  \mathfrak{E}_{k+2}[u](\tau)^{1/2}\right)^{\frac{r-2}{r}},
\end{cases} \label{p4rGNS-KG-ut-ver3} 
\\
   \tau^{1/r} \| u_t\|_{\mr \Sobw_{2r-5}^{k,2}(\Sigma_\tau)} & \lesssim \begin{cases}
  \left( \mathfrak{E}_k[u] (\tau)\right)^{2/r} \cdot\left( \mathfrak{E}_{k+1}[u_t] (\tau)^{1/2} \cdot  \mathfrak{E}_{k+2}[u_t](\tau)^{1/2}\right)^{\frac{r-2}{r}}, \\
 \left( \mathfrak{E}_k[u] (\tau)\right)^{2/r} \cdot \left( \mathfrak{E}_{k+1}[u] (\tau)^{1/2} \cdot  \mathfrak{E}_{k+2}[u_t](\tau)^{1/2}\right)^{\frac{r-2}{r}}.
\end{cases}  \label{p4rGNS-KG-ut-ver4}
\end{align}

\end{prop}

\begin{proof}
The proofs of these estimates are treated in the same way as the proof of Proposition \ref{KG-d=3}, so we merely highlight the differences. For estimates \eqref{pqrGNS-KGder-d=4-ver1} - \eqref{pqrGNS-KG-ut-d=4-ver4} we solve \eqref{eq:wtpqrGNS} with
\begin{gather*}
	d = 4\\
	q = 2\\ 
	p = 2\\
	1 + \beta q = \mu\\
	1-p + \alpha p =  \nu\\
	1/r = \theta /q + (1-\theta) /p^*,
\end{gather*}
where $\mu,\nu$ can again take the values $\pm 1$. Denoting the weight 
\[\rho_{pqr}(\mu,\nu) \eqdef 1 + (\theta \beta + (1-\theta)\cdot \alpha)r,\]
the interpolation inequality yields
\begin{equation}
\tau^{1/2 - 1/r} \| L^\alpha v\|_{\Lebw_{\rho_{pqr}(\mu,\nu)}^r} \lesssim \left( \| L^ \alpha v\|_{\Lebw_\mu^2}\right)^{\frac{4-r}{r}}\cdot \left( \| L^\alpha v\|_{\Sobw^{1,2}_\nu}\right)^{\frac{2r-4}{r}}. \label{KGder-d=4}
\end{equation}
One explicitly computes the weights as 
\[ \rho_{pqr}(1,-1) = 1, \quad \rho_{pqr}(1,1) = 2r-3, \quad \rho_{pqr}(-1,-1) = r-3, \quad \rho_{pqr}(-1,1) = 3r-7.\]
Replacing $\mu,\nu = \pm 1$ and $v \mapsto (u,u_t)$ or $(L^i u,u_t)$ in \eqref{KGder-d=4} then proves \eqref{pqrGNS-KGder-d=4-ver1} - \eqref{pqrGNS-KG-ut-d=4-ver4} by following the same analysis as in the proof of Proposition \ref{KG-d=3}.

For the estimates when $r > 4$, we appeal to the borderline \eqref{eq:pdrGNAWS} inequality with 
\begin{gather*}
		d = 4,\\
		q = 2,\\
		1 + \beta q + \alpha q = \sigma,\\
		(1-d)(1 + \theta\beta r) + \alpha d = \rho_{pq4}(\mu,\nu),\\
		1/r = \theta/q + (1-\theta)/(r + 1^*),
\end{gather*}
where $\sigma$ can take the values $\pm 1$ and $\rho_{pq4}(\mu,\nu)$ is as above.  Denoting the weight 
\[\rho_{d=4}(\sigma,\mu,\nu) \eqdef 1 + \theta\beta r + \alpha r,\] the borderline inequality yields
\begin{equation}
\left(\tau^{1/4}\right)^{\frac{r-2}{r}}\|L^\alpha v\|_{\Lebw_{\rho_{d=4}(\sigma,\mu,\nu)}^r}\lesssim \left( \|L^\alpha v\|_{\Lebw_\sigma^2}\right)^{2/r} \cdot \left( \| L^\alpha v\|_{\Sobw_{\rho_{pq4}(\mu,\nu)}^{1,4}}\right)^{\frac{r-2}{r}}.\label{KGder-d=4-border}
\end{equation}
This inequality is valid for $r > 2$ so in particular $r > 4$. One explicitly computes the weights 
\begin{align*}
&\rho_{d=4}(1,1,-1)  = r-1, && \rho_{d=4}(-1,1,-1)  = r-3, \\
&\rho_{d=4}(1,1,1)  = 2r-3, && \rho_{d=4}(-1,1,1) = 2r-5, \\
&\rho_{d=4}(1,-1,-1)  = r-1, & &\rho_{d=4}(-1,-1,-1)  =r-3, \\
&\rho_{d=4}(1,-1,1)  = 2r-3, & &\rho_{d=4}(-1,-1,1) = 2r-5.
\end{align*}
Replacing $\sigma,\mu,\nu = \pm 1$ and $v \mapsto (u,u_t)$ or $(L^i u,u_t)$ in \eqref{KGder-d=4} then proves \eqref{p4rGNS-KGder-ver1} - \eqref{p4rGNS-KG-ut-ver4} by following the same analysis as in the proof of Proposition \ref{KG-d=3}.
\end{proof}
\begin{rmk}
We note that even though the estimates for $r > 4$ in Proposition \ref{KG-d=4} had almost the same proofs as the ones for $r > 6$ in Proposition \ref{KG-d=3}, there is a notable difference between the two: there are {only} \emph{four} distinct weights for $\rho_{d=4}(\pm1 , \pm 1, \pm 1)$ while there are \emph{six} distinct weights for $\rho_{d=3}(\pm 1,\pm 1, \pm 1)$. The reason for this is that we controlled the second factor of \eqref{KGder-d=3} using the non-borderline estimates derived from \eqref{eq:wtpqrGNS} with $3 \in [2,6]$. On the other hand, the second factor of $\eqref{KGder-d=4}$ was estimated with the \emph{end point} $4 \in [2,4]$. 
\end{rmk}

\section{A nonlinear application}
In a previous paper we studied the stability of traveling wave solutions to the membrane equation \cite{Planewaves}. Key to our understanding there is the study of the following semilinear problem. 
\begin{equation}\label{eq:semi:model}
	\Box \phi = \Upsilon(t - x^1) (\partial_t \phi + \partial_{x^1} \phi)^2,
\end{equation}
where $\Upsilon \in C^\infty_0(\R)$ is arbitrary. We wish to study here the small-data Cauchy problem for \eqref{eq:semi:model} on $\R^{1+d}$ with $d = 2,3$, where for convenience we will prescribe the data at $t = 2$, such that
\begin{equation}
	\phi(2,x) = \phi_0(x), \quad \partial_t \phi(2,x) = \phi_1(x)
\end{equation}
for some $\phi_0,\phi_1\in C^\infty_0(B(0,1))$. 
For convenience of notation we will write $v = t - x^1$. 
Note that 
\[ L^1 v = -v; \quad L^i v = x^i.\]

By standard local existence theory and finite speed of propagation we can assume that for sufficiently small initial data, the solution exists up to $\Sigma_2$. The breakdown criterion for the wave equation implies that so long as we can show that the first derivatives $|L^i\phi|$ and $|\partial_t\phi|$ remain bounded on $\Sigma_\tau$ for all $\tau> 2$, we can guarantee global existence of solutions. 
A sufficient condition for global existence is therefore \emph{a priori bounds on the second-order energies}, in view of the (Morrey-) Sobolev inequalities such as those described in \cite{Wong2017} and recalling we fixed $d = 2,3$. 

Following our previous work \cite[Sections 4 and 5]{Planewaves} we will study the prolonged system satisfied by both $\phi$ and its derivative $L^1\phi$. 
First, observe that 
\[ \partial_t \phi + \partial_{x^1}\phi = \frac{1}{t} L^1\phi + \frac{v}{t} \partial_t \phi \]
Since $L^1$ is Killing, we see that after a small computation
\begin{multline*}
	\Box (L^1\phi) = -(v \Upsilon'(v) + \Upsilon(v)) \cdot \frac{1}{t^2} (L^1 \phi + v \partial_t \phi)^2 \\
	+ \frac{1}{t^2} \Upsilon(v) (L^1 \phi + v\partial_t\phi) (L^1 L^1\phi + \partial_t L^1 \phi).
\end{multline*}
Writing $\psi = L^1\phi$, then we are down to considering the following system of nonlinear wave equations
\begin{equation}
	\left\{
		\begin{aligned}
			\Box \phi &= \frac{1}{t^2} \Upsilon(v) (\psi + v \phi_t)^2; \\
			\Box \psi &= - \frac{1}{t^2} (v\Upsilon'(v) + \Upsilon(v)) (\psi + v\phi_t)^2 + \frac{1}{t^2} (\psi + v \phi_t) (L^1\psi + v \psi_t).
		\end{aligned}
	\right.
\end{equation}

Next, letting $v_<$ and $v_>$ be real numbers such that $\supp \Upsilon \subset [v_<, v_>]$, we can define as in \cite[Section 5]{Planewaves} the schematic notation $\mathscr{P}_k$ which will stand for any arbitrary function $f$ on $\R^{1+d}$ satisfying
\begin{itemize}
	\item $\supp f \subset \{ v\in [v_<, v_>]\}$; and
	\item restricting to the forward light-cone $\{t > |x|+1\}$, we have the uniform bound $|f| \lesssim \tau^k$. 
\end{itemize}
Quite clearly if $f(t,x) = \chi(v)$ (where $\chi$ is any compactly supported smooth function), then $f = \mathscr{P}_0$. 
By the computations in \cite[Sections 3.2, 6.2]{Planewaves}, we further have that higher $L$ derivatives of $f$ are
\begin{equation}
	L^\alpha f = \mathscr{P}_{|\alpha|}.
\end{equation}
Note, as we saw already in the derivation of our prolonged system, $L^1\Upsilon(v) = -v \Upsilon'(v) = \mathscr{P}_0$, so the above bound is not optimal when differentiating in $L^1$. 

\begin{rmk}[Bounds for $w_\tau$]
	Notice that on the subset $\{ v\in [v_<, v_>]\} \cap \{t > |x| + 1\}$, we have the following comparison
	\begin{equation}\label{eq:ttau:comp}
		\tau^2 \approx w_\tau.
	\end{equation}
	We will make use of this throughout.
\end{rmk}

Therefore if we apply the weighted commutator algebra developed in \cite[Section 3.2]{Planewaves}, we see that the higher derivatives of $\phi$ and $\psi$ satisfy the following system of differential inequalities:
\begin{equation}\label{eq:main:NLsystem}
	\left\{
		\begin{aligned}
			\bigl|\Box L^\alpha \phi\bigr| &\lesssim \frac{1}{t^2} \sum_{k + |\beta| + |\gamma| \leq |\alpha|} \bigl| \mathscr{P}_k \cdot(L^\beta \psi + \partial_t (L^\beta \phi))\cdot(L^\gamma \psi + \partial_t (L^\gamma\phi)) \bigr|;\\
			\bigl|\Box L^\alpha \psi\bigr| & \lesssim \frac{1}{t^2} \sum_{k + |\beta| + |\gamma| \leq |\alpha|} \bigl| \mathscr{P}_k \cdot (L^\beta \psi + \partial_t (L^\beta \phi))\cdot (L^\gamma L^1\psi + \partial_t (L^\gamma\psi))\bigr|.
		\end{aligned}
	\right.
\end{equation}
We remark that the coordinate function $t$ is equivalent to $w_\tau$ when restricted to $\Sigma_\tau$ (by definition). 
Below we will discuss the \textit{a priori} estimates that can be proven for the system \eqref{eq:main:NLsystem}. 
Specifically, we will describe the improvements that can be made as a consequence of the interpolation inequalities described in Section \ref{sect:lin:est}.

\subsection{The basic energy estimates}
We will denote by 
\[ \mathcal{E}_\tau[u] = \int_{\Sigma_{\tau}} \frac{1}{\tau w_\tau} \sum |L^i u|^2 + \frac{\tau}{w_\tau} |\partial_tu|^2 ~\dvol = \frac{1}{\tau} \|u\|_{\mr\Sobw_{-1}^{1,2}(\Sigma_\tau)}^2 + \tau \|u_t\|_{\Lebw_{-1}^2(\Sigma_\tau)}^2\]
where $u\in \R^{1+d}\to\R$. This energy integral satisfies the basic energy inequality for wave equations: if $\tau_0 < \tau_1$ we have
\begin{equation}
	\mathcal{E}_{\tau_1}[u] - \mathcal{E}_{\tau_0}[u] \lesssim \int_{\tau_0}^{\tau_1}\int_{\Sigma_\tau} |\Box u| \cdot |u_t| \dvol_{\Sigma_\tau} \D\tau.
\end{equation}
By Cauchy-Schwarz, we then have 
\begin{equation}
	\mathcal{E}_{\tau_1}[u] - \mathcal{E}_{\tau_0}[u] \lesssim \int_{\tau_0}^{\tau_1} \tau^{-1/2} \|\Box u\|_{\Lebw_1^2(\Sigma_\tau)} \mathcal{E}_{\tau}^{1/2}[u] ~\D\tau.
\end{equation}

Returning to \eqref{eq:main:NLsystem}, let us introduce the notations
\begin{gather}
	\mathfrak{E}_k(\tau) := \tau^{-1/2} \sum_{j = 1}^{k+1} \|\phi\|_{\mr\Sobw_{-1}^{j,2}(\Sigma_\tau)} + \tau^{1/2} \|\phi_t \|_{\Sobw_{-1}^{k,2}(\Sigma_\tau)},\\
	\mathfrak{F}_k(\tau) := \tau^{-1/2} \|\psi\|_{\Sobw_{-1}^{k+1,2}(\Sigma_\tau)} + \tau^{1/2} \|\psi_t \|_{\Sobw_{-1}^{k,2}(\Sigma_\tau)}.
\end{gather}
Using the commutator algebra properties (specifically those of $[L^i,\partial_t]$) described in \cite[Section 3.2]{Planewaves}, we see that 
\begin{gather*}
	(\mathfrak{E}_k(\tau))^2 \approx \sum_{|\alpha| \leq k} \mathcal{E}_\tau[L^\alpha \phi], \\
	(\mathfrak{F}_k(\tau))^2 \lesssim \sum_{|\alpha| \leq k} \mathcal{E}_\tau[L^\alpha \psi] + \mathcal{E}_\tau[\phi],\\
	\sum_{|\alpha| \leq k} \mathcal{E}_\tau[L^\alpha \psi] \lesssim (\mathfrak{F}_k(\tau))^2.
\end{gather*}
So our fundamental energy estimates read as
\begin{gather}
	\label{eq:energest:e}\mathfrak{E}_k(\tau_1)^2 - \mathfrak{E}_k(\tau_0)^2 \lesssim \int_{\tau_0}^{\tau_1} \tau^{-1/2} \|\Box \phi\|_{\Sobw_1^{k,2}} \mathfrak{E}_k(\tau) ~\D\tau,\\
	\label{eq:energest:f}\mathfrak{F}_k(\tau_1)^2 - \mathfrak{F}_k(\tau_0)^2 \lesssim \int_{\tau_0}^{\tau_1} \tau^{-1/2} \|\Box \psi\|_{\Sobw_1^{k,2}} \mathfrak{F}_k(\tau) ~\D\tau + \int_{\tau_0}^{\tau_1} \tau^{-1/2} \|\Box \phi\|_{\Lebw_1^{2}} \mathfrak{E}_0(\tau) ~\D\tau.
\end{gather}
Finally, using \eqref{eq:main:NLsystem}, we can estimate the inhomogeneities by
\begin{equation}\label{eq:main:NLest1}
	\left\{
		\begin{aligned}
			\| \Box\phi\|_{\Sobw_1^{k,2}} &\lesssim \sum_{j + |\beta| + |\gamma| \leq k} \Bigl\| \mathscr{P}_j \cdot (L^\beta\psi + \partial_t (L^\beta \phi)) \cdot (L^\gamma \psi + \partial_t(L^\gamma\phi)) \Bigr\|_{\Lebw_{-3}^{2}}, \\
			\| \Box\psi\|_{\Sobw_1^{k,2}} &\lesssim \sum_{j + |\beta| + |\gamma| \leq k} \Bigl\| \mathscr{P}_j \cdot (L^\beta\psi +          \partial_t (L^\beta \phi)) \cdot (L^\gamma L^1\psi + \partial_t(L^\gamma\psi)) \Bigr\|_{\Lebw_{-3}^{2}}.
		\end{aligned}
	\right.
\end{equation}
Note that we have absorbed the $t^{-2}$ weight into the weighted Lebesgue space on the right. 

\subsection{The bootstrap using only Morrey-Sobolev-type estimates}
In this section we will estimate the terms in \eqref{eq:main:NLest1} using only the Morrey-Sobolev-type inequality \eqref{Morrey}, when dimension $d = 2$ or $3$. 
In these cases we have
\[ \tau^{1/2} \| w_\tau^{d/2-1} u\|_{L^\infty(\Sigma_\tau)} \lesssim \|u\|_{\Sobw_{-1}^{2,2}(\Sigma_\tau)}.\]
First we treat the nonlinearity for $\Box \phi$. By symmetry we can assume that $|\beta| \leq |\gamma|$ in \eqref{eq:main:NLest1}. 
This implies
\begin{multline*}
	\Bigl\| \mathscr{P}_j \cdot (L^\beta\psi + \partial_t (L^\beta \phi)) \cdot (L^\gamma \psi + \partial_t(L^\gamma\phi)) \Bigr\|_{\Lebw_{-3}^{2}}\\
	\lesssim \tau^{j-d} \| w_\tau^{d/2 - 1} (L^\beta \psi + \partial_t (L^\beta \phi)) \|_{L^\infty} \| L^\gamma \psi + \partial_t (L^\gamma\phi) \|_{\Lebw_{-1}^2}
\end{multline*}
in which derivation we freely used \eqref{eq:ttau:comp}. Our Morrey-type inequality implies then
\[ \lesssim \tau^{j-d} \left( \mathfrak{F}_{1+|\beta|}(\tau) + \tau^{-1} \mathfrak{E}_{2+|\beta|}(\tau)\right)\left( \tau^{1/2} \mathfrak{F}_{|\gamma| - 1}(\tau) + \tau^{-1/2} \mathfrak{E}_{|\gamma|}(\tau) \right) \]
when $|\gamma| > 0$. When $|\beta| = |\gamma| = 0$ we have instead 
\[ \lesssim \tau^{j-d+1/2} \left( \mathfrak{F}_{1}(\tau) + \tau^{-1} \mathfrak{E}_{2}\right) \mathfrak{E}_0(\tau).\]
We summarize our result in the following proposition.
\begin{prop}[Estimates for $\Box \phi$] \label{prop:morrey:boxphi}
	Fix $d = 2$ or $3$, then 
	\begin{equation}
		\|\Box \phi\|_{\Lebw_{1}^2} \lesssim \tau^{1/2 - d} (\mathfrak{F}_1(\tau) + \tau^{-1} \mathfrak{E}_2(\tau)) \mathfrak{E}_0(\tau).
	\end{equation}
	When $k > 0$ we also have
	\begin{multline}
		\| \Box\phi\|_{\Sobw_1^{k,2}} \lesssim \tau^{k + 1/2 - d} (\mathfrak{F}_1(\tau) + \tau^{-1}\mathfrak{E}_2(\tau)) \mathfrak{E}_0(\tau) \\
		+ \sum_{j = 0}^{k-1} \sum_{\ell = \lceil (k-j)/2 \rceil}^{k-j} \tau^{j - d}\left(\mathfrak{F}_{1 + k - j - \ell}(\tau) + \tau^{-1} \mathfrak{E}_{2+k - j - \ell}(\tau)\right)\left(\tau^{1/2}\mathfrak{F}_{\ell - 1}(\tau) + \tau^{-1/2} \mathfrak{E}_{\ell}(\tau)\right).
	\end{multline}
\end{prop}

Similarly we can analyze the nonlinearity for $\Box \psi$. We split into two cases: first with $|\gamma| \geq |\beta|$, and second with $|\gamma| < |\beta|$. 
In the first case, we have 
\begin{multline*}
	\Bigl\| \mathscr{P}_j \cdot (L^\beta\psi + \partial_t (L^\beta \phi)) \cdot (L^\gamma L^1 \psi + \partial_t(L^\gamma\psi)) \Bigr\|_{\Lebw_{-3}^{2}}\\
	\lesssim \tau^{j-d} \| w_\tau^{d/2 - 1} (L^\beta \psi + \partial_t (L^\beta \phi)) \|_{L^\infty} \| L^\gamma L^1 \psi + \partial_t (L^\gamma\psi) \|_{\Lebw_{-1}^2}
\end{multline*}
which leads us to 
\[ \lesssim \tau^{j-d + 1/2} \left( \mathfrak{F}_{1 + |\beta|}(\tau) + \tau^{-1} \mathfrak{E}_{2 + |\beta|}(\tau) \right) \mathfrak{F}_{|\gamma|}(\tau).
\]
For the second case, we have
\begin{multline*}
	\Bigl\| \mathscr{P}_j \cdot (L^\beta\psi + \partial_t (L^\beta \phi)) \cdot (L^\gamma L^1 \psi + \partial_t(L^\gamma\psi)) \Bigr\|_{\Lebw_{-3}^{2}}\\
	\lesssim \tau^{j-d} \| L^\beta \psi + \partial_t (L^\beta \phi) \|_{\Lebw_{-1}^2} \| w_{\tau}^{d/2 - 1} (L^\gamma L^1 \psi + \partial_t (L^\gamma\psi)) \|_{L^\infty}
\end{multline*}
which leads us to 
\[ \lesssim \tau^{j-d} \left( \tau^{1/2} \mathfrak{F}_{|\beta|-1}(\tau) + \tau^{-1/2} \mathfrak{E}_{|\beta|}(\tau) \right) \mathfrak{F}_{|\gamma|+2}(\tau).
\]
These can be summarized in the following proposition.
\begin{prop}[Estimate for $\Box \psi$] \label{prop:morrey:boxpsi}
	Fix $d = 2$ or $3$, then
	\begin{multline}
		\| \Box \psi\|_{\Sobw_1^{k,2}} \lesssim \sum_{j = 0}^k \sum_{\ell = \lceil (k-j)/2\rceil}^{k-j} \tau^{j - d + 1/2} 
		\left( \mathfrak{F}_{1 + k - j - \ell}(\tau) + \tau^{-1} \mathfrak{E}_{2 + k - j - \ell}(\tau) \right) \mathfrak{F}_{\ell}(\tau)\\
		+ \sum_{j = 0}^{k-1} \sum_{\ell = \lfloor (k-j)/2 \rfloor + 1}^{k - j} \tau^{j-d + 1/2} \left( \mathfrak{F}_{\ell - 1}(\tau) + \tau^{-1} \mathfrak{E}_{\ell}(\tau) \right) \mathfrak{F}_{2 + k - j - \ell}(\tau).
	\end{multline}
\end{prop}

Based on the two propositions above, we can close the bootstrap argument for global existence with polynomially growing energies. More precisely, we have the following two theorems. 
\begin{thm}[$d = 3$ GWP bootstrap using Morrey] \label{thm:d3:morrey}
	Fix $d = 3$ and $\kappa \geq 3$. Assume the initial data satisfies
	\begin{equation}
		\mathfrak{E}_\kappa(2), \mathfrak{F}_\kappa(2) \leq \epsilon 
	\end{equation}
	and that for some $T > 2$, the bootstrap assumptions
	\begin{equation}
		\left\{
			\begin{gathered}
				\mathfrak{E}_0(\tau), \mathfrak{E}_1(\tau), \mathfrak{F}_0(\tau), \mathfrak{F}_1(\tau) \leq \delta \\
				\mathfrak{E}_2(\tau), \mathfrak{F}_2(\tau) \leq \delta \ln(\tau)\\
				\mathfrak{E}_k(\tau), \mathfrak{F}_k(\tau) \leq \delta \tau^{k-2}, \quad 3 \leq k \leq \kappa
			\end{gathered}
		\right.
	\end{equation}
	hold for all $\tau\in [2,T]$. Then there exists some constant $C$ which \emph{depends only on the background $\Upsilon$ and the number of derivatives $\kappa$}, such that the improved estimates
	\begin{equation}
		\left\{
			\begin{gathered}
				\mathfrak{E}_0(\tau), \mathfrak{E}_1(\tau), \mathfrak{F}_0(\tau), \mathfrak{F}_1(\tau) \leq \epsilon + C \delta^{3/2} \\
				\mathfrak{E}_2(\tau), \mathfrak{F}_2(\tau) \leq \epsilon + C \delta^{3/2} \ln(\tau) \\
				\mathfrak{E}_k(\tau), \mathfrak{F}_k(\tau) \leq \epsilon + C \delta^{3/2} \tau^{k-2}, \quad 3 \leq k \leq \kappa
			\end{gathered}
		\right.
	\end{equation}
	hold on $\tau \in [2,T]$. 
\end{thm}

\begin{rmk}
	The lower bound $\kappa \geq 3$ is chosen so that between the energy estimates \eqref{eq:energest:e} and \eqref{eq:energest:f}, and the nonlinear estimates Propositions \ref{prop:morrey:boxphi} and \ref{prop:morrey:boxpsi}, we have a closed system. 
\end{rmk}

\begin{proof}
	Applying the bootstrap assumptions to Proposition \ref{prop:morrey:boxphi} we get that 
	\begin{gather*}
		\tau^{-1/2} \|\Box\phi\|_{\Lebw_{1}^2} \lesssim \delta^2 \tau^{-3}, \\
		\tau^{-1/2} \|\Box\phi\|_{\Sobw_{1}^{k,2}} \lesssim \delta^2 \tau^{k-3} + \sum_{j = 0}^{k-1} \sum_{\ell = \lceil (k-j)/2\rceil}^{k-j} \delta^2 \tau^{j-3 + \max(k - j - \ell - 2,0) + \max(\ell - 3,0)} \ln(\tau)^2.
	\end{gather*}
	Observe that
	\begin{multline*} 
		j - 3 + \max( k-j-\ell - 2, 0) + \max(\ell - 3,0) \\
		= \max( k - 8, j + \ell - 6, k-\ell-5, j-3) \leq k - 4
	\end{multline*}
	using that $j \leq k - 1$, and $j + \ell \leq k$, we conclude that for every $k \geq 0$, 
	\begin{equation}
		\tau^{-1/2} \|\Box \phi\|_{\Sobw_{1}^{k,2}} \lesssim \delta^2 \tau^{k-3}.
	\end{equation}
	The improved estimates for $\mathfrak{E}_*$ follows from \eqref{eq:energest:e}. 

	Similarly we can apply the bootstrap assumption to Proposition \ref{prop:morrey:boxpsi} and we get that
	\begin{multline*}
		\tau^{-1/2} \|\Box\psi\|_{\Sobw_1^{k,2}} \lesssim \delta^2 \tau^{k-3} + \sum_{j = 0}^{k-1}\sum_{\ell = \lceil (k-j)/2\rceil}^{k-j} \delta^2 \tau^{j-3} \tau^{\max(k - j - \ell-1,0)} \tau^{\max(\ell - 2,0)} \ln(\tau)^2 \\
		+ \sum_{j = 0}^{k-1} \sum_{\ell = \lfloor (k-j)/2\rfloor + 1}^{k-j} \delta^2 \tau^{j-3} \tau^{\max(\ell - 3,0)} \tau^{\max(k - j - \ell,0)} \ln(\tau)^2.
	\end{multline*}
	Arguing similarly as before we have, for $j \leq k - 1$
	\begin{multline*}
		j - 3 + \max(k - j - \ell - 1, 0) + \max(\ell - 2,0) \\
		= \max(k - 6, k - \ell - 4, j + \ell - 5, j - 3) \leq k - 4.
	\end{multline*}
	and, when $j \leq k - 1$ and $\ell \geq 1$
	\begin{multline*}
		j - 3 + \max(\ell - 3,0) + \max(k - j - \ell,0) \\
		= \max( k - 6, k - \ell - 3, j + \ell - 6, j - 3) \leq k - 4.
	\end{multline*}
	This implies
	\begin{equation}
		\tau^{-1/2} \|\Box \phi\|_{\Sobw_1^{k,2}} \lesssim \delta^2 \tau^{k-3}
	\end{equation}
	and the improved estimates for $\mathfrak{F}_*$ follows from \eqref{eq:energest:f}. 
\end{proof}

The case for $d = 2$ is worse, due to certain interaction terms that appear. 
Specifically, let us consider the estimates first for $\|\Box\psi\|_{\Sobw_{1}^{1,2}}$. From Proposition \ref{prop:morrey:boxpsi} we see 
\begin{multline*} 
	\tau^{-1/2} \|\Box\psi\|_{\Sobw_1^{1,2}} \lesssim \tau^{-2} (\mathfrak{F}_1(\tau) + \tau^{-1} \mathfrak{E}_2(\tau)) \mathfrak{F}_1(\tau)\\
	+ \tau^{-1} (\mathfrak{F}_1(\tau) + \tau^{-1}\mathfrak{E}_2(\tau)) \mathfrak{F}_0(\tau) 
	+ \tau^{-2} (\mathfrak{F}_0(\tau) + \tau^{-1}\mathfrak{E}_1(\tau))\mathfrak{F}_2(\tau).
\end{multline*}
The presence of a term of the form $\tau^{-1} \mathfrak{F}_1(\tau) \mathfrak{F}_0(\tau)$ on the right indicates that the \emph{best} one can hope for in terms of a bound for the energy $\mathfrak{F}_1(\tau)$ is $\delta \tau^\gamma$ for some small $\gamma$. 
The fact that the bound by Proposition \ref{prop:morrey:boxphi} for $\tau^{-1/2} \|\Box\phi\|_{\Sobw_1^{k,2}}$ has a term of the form $\tau^{k - 2} \mathfrak{F}_1(\tau) \mathfrak{E}_0(\tau)$ signals that the best we can hope for $\mathfrak{E}_k(\tau)$ in general is $\delta \tau^{k - 1 + \gamma}$, whenever $k \geq 1$. This is a significantly heavier loss compared to the $d = 3$ case presented above. 

\begin{thm}[$d = 2$ GWP bootstrap using Morrey] \label{thm:d2:morrey}
	Fix $d = 2$ and $\kappa \geq 3$, as well as $\gamma \in (0,1/3)$. Assume the initial data satisfies
	\begin{equation}
		\mathfrak{E}_\kappa(2), \mathfrak{F}_\kappa(2) \leq \epsilon
	\end{equation}
	and that for some $T > 2$, the bootstrap assumptions
	\begin{equation}
		\left\{
			\begin{gathered}
				\mathfrak{E}_0(\tau), \mathfrak{F}_0(\tau) \leq \delta\\
				\mathfrak{E}_k(\tau), \mathfrak{F}_k(\tau) \leq \delta \tau^{k-1 + \gamma}, \quad 1 \leq k \leq \kappa
			\end{gathered}
		\right.
	\end{equation}
	hold for all $\tau\in [2,T]$. Then there exists a constant $C$ depending only on $\kappa, \gamma$ and the initial profile $\Upsilon$, such that the improved estimates
	\begin{equation}
		\left\{
			\begin{gathered}
				\mathfrak{E}_0(\tau), \mathfrak{F}_0(\tau) \leq \epsilon  + C \delta^{3/2}\\
				\mathfrak{E}_k(\tau), \mathfrak{F}_k(\tau) \leq \epsilon + C \delta^{3/2} \tau^{k-1+\gamma}, \quad 1 \leq k \leq \kappa
			\end{gathered}
		\right.
	\end{equation}
	hold on $\tau\in [2,T]$. 
\end{thm}
\begin{proof}
	Again we will first apply our bootstrap assumptions to Proposition \ref{prop:morrey:boxphi}. 
	This shows that
	\begin{gather*}
		\tau^{-1/2} \|\Box \phi\|_{\Lebw_1^2} \lesssim \delta^2 \tau^{\gamma-2},\\
		\tau^{-1/2} \|\Box \phi\|_{\Sobw_1^{k,2}} \lesssim \delta^2 \tau^{k-2 + \gamma} + \delta^2 \sum_{j = 0}^{k-1} \sum_{\ell = \lceil (k-j)/2\rceil}^{k-j} \tau^{j-2 + \max(k-j-\ell + \gamma,0) + \max(\ell - 2 + \gamma,0)}.
	\end{gather*}
	Arguing as before
	\begin{multline*}
		j - 2 + \max(k-j-\ell + \gamma,0) + \max(\ell - 2 + \gamma, 0) \\
		= \max(k - 4 + 2\gamma, k - \ell - 2 + \gamma , j + \ell - 4 + \gamma, j - 2) \leq k - 3 + \gamma
	\end{multline*}
	noting that since $j < k$ we have $\ell \geq 1$. So we conclude that
	\[ \tau^{-1/2} \|\Box\phi\|_{\Lebw_1^2} \lesssim \delta^2 \tau^{k + \gamma-2} \]
	and the improved estimates for $\mathfrak{E}_*(\tau)$ follows from \eqref{eq:energest:e}. 

	Similarly from Proposition \ref{prop:morrey:boxpsi} we get
	\begin{multline*}
		\tau^{-1/2} \|\Box\psi\|_{\Sobw_1^{k,2}} \lesssim \delta^2 \sum_{j = 0}^k \sum_{\ell = \lceil (k-j)/2\rceil}^{k-j} \tau^{j-2 + \max(k-j-\ell + \gamma,0) + \max(\ell - 1 + \gamma,0)} \\
		+ \delta^2 \sum_{j = 0}^{k-1} \sum_{\ell = \lfloor (k-j)/2\rfloor + 1}^{k-j} \tau^{j-2 + \max(\ell - 2 + \gamma,0) + \max(1 + k - j - \ell + \gamma,0)}.
	\end{multline*}
	For the first exponent we have
	\begin{multline*}
		j - 2 + \max(k-j-\ell + \gamma,0) + \max(\ell - 1 + \gamma,0) \\
		= \max(k - 3 + 2\gamma, k - \ell - 2 + \gamma, j + \ell - 3 + \gamma, j - 2) \leq k - 2 + \gamma
	\end{multline*}
	using now that $\ell$ may be zero in our case. For the second exponent we have
	\begin{multline*}
		j - 2 + \max(\ell - 2 + \gamma, 0) + \max(1 + k - j - \ell + \gamma, 0) \\
		= \max(k - 3 + 2\gamma, j + \ell - 4 + \gamma, k - 1 - \ell + \gamma, j - 2) \leq k - 2 + \gamma
	\end{multline*}
	using for this second exponent we have a lower bound $\ell \geq 1$. 
	Applying the energy estimate \eqref{eq:energest:e} we have the improved estimates for $\mathfrak{F}_*(\tau)$. 
\end{proof}

\subsection{The bootstrap using Gagliardo-Nirenberg-Sobolev-type estimates}
It turns out that with the aid of the Gagliardo-Nirenberg-Sobolev-type estimates, in $d = 2$ we can rid ourselves of (most of) the $\gamma$ loss, and obtain an energy hierarchy more akin to what is shown in Theorem \ref{thm:d3:morrey} for the $d = 3$ case.
We expect that this improvement will also allow us to close our argument for the original \emph{quasilinear} problem in $d = 2$. 
We note here that the $d = 2$ quasilinear problem has also been treated by Liu and Zhou \cite{LiuZhou2019p} using different methods.
We defer a detailed discussion of the $d = 2$ quasilinear problem to a future manuscript, and focus here on the improvements one can make to the semilinear problem.  

The main improvement of using the Gagliardo-Nirenberg-Sobolev-type inequalities over the Morrey-type inequalities for our application lies in Remark \ref{rmk:morrey}. In $d = 2$, compared to scaling, the Morrey-type inequality \eqref{Morrey} loses \emph{one whole derivative}.
For traditional applications of the vector field method this loss is inconsequential, as energies to all orders are comparable (they are generally always all bounded, with possible the exception of the borderline top order energies). 
In our setting, however, our equation forces us to consider an energy hierarchy with polynomial growth rates. So this loss of one derivative carries a corresponding loss of \emph{decay}, which manifests as the $\tau^\gamma$ loss in the energy hierarchy in Theorem \ref{thm:d2:morrey} compared to Theorem \ref{thm:d3:morrey}. 
This loss can be overcome using the Gagliardo-Nirenberg-Sobolev-type inequalities which are scaling sharp, which converts our losses into merely a $\ln(\tau)$ one.  

In the course of the proof we will need the following simple lemma:
\begin{lem}\label{lem:tech}
	Fix $p > 0$. There exists a constant $C$ (depending on $p$) such that 
	\[ \int_1^\tau \sigma^{p-1} \ln(\sigma)^2 ~\D\sigma \leq C \tau^p \ln(\tau)^2.\]
\end{lem}
\begin{proof}
	The lemma follows immediately from the differential identity 
	\[
		\frac{d}{dx} \left( x^p \ln(x)^2 - \frac{2}{p} x^p \ln(x) + \frac{2}{p^2} x^p\right) = p x^{p-1} \ln(x)^2.
	\]
	In fact we can take $C = \frac{1}{p} + \frac{2}{p^3}$. 
\end{proof}

\begin{thm}[$d = 2$ GWP bootstrap using GNS]
	Fix $d = 2$ and $\kappa \geq 2$. Assume the initial data satisfies 
	\begin{equation}
		\mathfrak{E}_\kappa(2), \mathfrak{F}_\kappa(2) \leq \epsilon
	\end{equation}
	and that for some $T > 2$, the bootstrap assumptions 
	\begin{equation}
		\left\{
			\begin{gathered}
				\mathfrak{E}_0(\tau), \mathfrak{F}_0(\tau) \leq \delta\\
				\mathfrak{E}_1(\tau) \leq \delta \ln(\tau)\\
				\mathfrak{F}_1(\tau) \leq \delta \ln(\tau)^2\\
				\mathfrak{E}_k(\tau) \leq \delta \tau^{k - 1}, \qquad 2 \leq k \leq \kappa \\
				\mathfrak{F}_k(\tau) \leq \delta \tau^{k-1} \ln(\tau), \qquad 2 \leq k \leq \kappa
			\end{gathered}
		\right.
	\end{equation}
	hold for all $\tau\in [2,T]$. Then there exists some constant $C$ which depends only on $\kappa$ and the background profile $\Upsilon$, such that the improved estimates
	\begin{equation}
		\left\{
			\begin{gathered}
				\mathfrak{E}_0(\tau), \mathfrak{F}_0(\tau) \leq \epsilon + C\delta^{3/2}\\
				\mathfrak{E}_1(\tau) \leq \epsilon + C\delta^{3/2} \ln(\tau)\\
				\mathfrak{F}_1(\tau) \leq \epsilon + C\delta^{3/2} \ln(\tau)^2\\
				\mathfrak{E}_k(\tau) \leq \epsilon + C\delta^{3/2} \tau^{k - 1}, \qquad 2 \leq k \leq \kappa\\
				\mathfrak{F}_k(\tau) \leq \epsilon + C\delta^{3/2} \tau^{k - 1} \ln(\tau), \qquad 2 \leq k \leq \kappa
			\end{gathered}
		\right.
	\end{equation}
	hold on $\tau\in [2,T]$. 
\end{thm}

\begin{proof}
	The key to our proof is to replace Propositions \ref{prop:morrey:boxphi} and \ref{prop:morrey:boxpsi} using the Gagliardo-Nirenberg-Sobolev inequalities instead of Morrey-type inequalities. 
	Rather naturally, since we have a quadratic term measured in $\Lebw_{-3}^2$ in \eqref{eq:main:NLest1}, we will put each of the quadratic terms in $\Lebw_{-1}^4$ and take advantage of the remaining decaying weights. 
	Recall here \eqref{p2rGNS-waveder}, for which we have set $r = 4$:
	\begin{gather}
		\tau^{-1/4} \left( \|\phi\|_{\mr\Sobw_{-1}^{k+1,4}} + \tau \|\partial_t\phi\|_{\mr\Sobw_{-1}^{k,4}} \right) \lesssim (\mathfrak{E}_k)^{1/2} (\mathfrak{E}_{k+1})^{1/2},\\
		\tau^{-1/4} \left( \|\psi\|_{\mr\Sobw_{-1}^{k+1,4}} + \tau \|\partial_t\psi\|_{\mr\Sobw_{-1}^{k,4}} \right) \lesssim               (\mathfrak{F}_k)^{1/2} (\mathfrak{F}_{k+1})^{1/2},\\
		\tau^{-1/4} \|\psi\|_{\Lebw_{-1}^4} \lesssim \mathfrak{F}_0.
	\end{gather}

	Let us now estimate $\Box\phi$. Returning to \eqref{eq:main:NLest1}, we will assume again that $|\beta| \leq |\gamma|$. The inhomogeneities give (where for convenience of notation we will set $\mathfrak{F}_{-1} \eqdef \mathfrak{F}_0$)
	\begin{multline*}
		\|\mathscr{P}_j w_\tau^{-1} (L^\beta\psi + \partial_tL^\beta\phi)(L^\gamma\psi + \partial_t L^\gamma\psi) \|_{\Lebw_{-1}^2} \\
		\lesssim \tau^{j-2} \bigl( \tau^{1/4} \mathfrak{F}_{|\beta|-1}^{1/2} \mathfrak{F}_{|\beta|}^{1/2} + \tau^{-3/4} \mathfrak{E}_{|\beta|}^{1/2} \mathfrak{E}_{|\beta|+1}^{1/2}\bigr) \bigl(\tau^{1/4} \mathfrak{F}_{|\gamma| - 1}^{1/2} \mathfrak{F}_{|\gamma|}^{1/2} + \tau^{-3/4} \mathfrak{E}_{|\gamma|}^{1/2} \mathfrak{E}_{|\gamma|+1}^{1/2}\bigr).
	\end{multline*}
	When $k$ is small, we can use purely this estimate to get
	\begin{multline}\label{eq:boxphi:gns:smallk}
		\tau^{-1/2} \|\Box\phi\|_{\Sobw_1^{k,2}} \lesssim \sum_{j + \ell \leq k} \tau^{j-2} \bigl(  \mathfrak{F}_{\ell-1}^{1/2} \mathfrak{F}_{\ell}^{1/2} + \tau^{-1} \mathfrak{E}_{\ell}^{1/2} \mathfrak{E}_{\ell+1}^{1/2}\bigr)\cdot\\
		\bigl(\mathfrak{F}_{k - j -\ell - 1}^{1/2} \mathfrak{F}_{k - j - \ell}^{1/2} + \tau^{-1} \mathfrak{E}_{k - j - \ell}^{1/2} \mathfrak{E}_{k - j - \ell+1}^{1/2}\bigr).
	\end{multline}
	We cannot close using only this estimate, as the right hand side depends on energies of order higher than $k$. For large $k$, we will isolate the borderline cases and use \eqref{Morrey} for those terms. This gives
	\begin{multline}\label{eq:boxphi:gns:bigk}
		\tau^{-1/2} \|\Box\phi\|_{\Sobw_1^{k,2}} \lesssim \tau^{-2} (\mathfrak{F}_1 + \tau^{-1}\mathfrak{E}_2)(\mathfrak{F}_{k-1} + \tau^{-1}\mathfrak{E}_k)\\
		+ \sum_{\ell = 1}^{k-1} \tau^{-2}  \bigl(  \mathfrak{F}_{\ell-1}^{1/2} \mathfrak{F}_{\ell}^{1/2} + \tau^{-1} \mathfrak{E}_{\ell}^{1/2} \mathfrak{E}_{\ell+1}^{1/2}\bigr)
		\bigl(\mathfrak{F}_{k  -\ell - 1}^{1/2} \mathfrak{F}_{k  - \ell}^{1/2} + \tau^{-1} \mathfrak{E}_{k - \ell}^{1/2} \mathfrak{E}_{k - \ell+1}^{1/2}\bigr)\\
		+ \sum_{j = 1}^k \sum_{\ell \leq k - j} \tau^{j-2} \bigl(  \mathfrak{F}_{\ell-1}^{1/2} \mathfrak{F}_{\ell}^{1/2} + \tau^{-1} \mathfrak{E}_{\ell}^{1/2} \mathfrak{E}_{\ell+1}^{1/2}\bigr)
		\bigl(\mathfrak{F}_{k - j -\ell - 1}^{1/2} \mathfrak{F}_{k - j - \ell}^{1/2} + \tau^{-1} \mathfrak{E}_{k - j - \ell}^{1/2} \mathfrak{E}_{k - j - \ell+1}^{1/2}\bigr).
	\end{multline}

	Similarly we can estimate $\Box \psi$ starting from \eqref{eq:main:NLest1}. The inhomogeneities give
	\begin{multline*}
		\|\mathscr{P}_j w_\tau^{-1} (L^\beta \psi + \partial_t L^\beta\phi)(L^\gamma L^1\psi + \partial_t L^\gamma \psi) \|_{\Lebw_{-1}^2} \\
		\lesssim \tau^{j-2+1/4} \bigl(\tau^{1/4} \mathfrak{F}_{|\beta|-1}^{1/2}\mathfrak{F}_{|\beta|}^{1/2} + \tau^{-3/4} \mathfrak{E}_{|\beta|}^{1/2} \mathfrak{E}_{|\beta|+1}^{1/2}\bigr)  \mathfrak{F}_{|\gamma|}^{1/2} \mathfrak{F}_{|\gamma|+1}^{1/2}
	\end{multline*}
	For small $k$ this implies the estimate
	\begin{equation}\label{eq:boxpsi:gns:smallk}
		\tau^{-1/2} \|\Box\psi\|_{\Sobw_1^{k,2}} \lesssim \sum_{j + \ell \leq k} \tau^{j-2} \bigl( \mathfrak{F}_{\ell - 1}^{1/2} \mathfrak{F}_{\ell}^{1/2} + \tau^{-1} \mathfrak{E}_{\ell}^{1/2} \mathfrak{E}_{\ell+1}^{1/2} \bigr) \mathfrak{F}_{k - j - \ell}^{1/2} \mathfrak{F}_{k - j - \ell + 1}^{1/2}.
	\end{equation}
	For large $k$ we have to also handle the top-order borderline terms differently, using Morrey. This gives
	\begin{multline}\label{eq:boxpsi:gns:bigk}
		\tau^{-1/2} \|\Box\psi\|_{\Sobw_1^{k,2}} \lesssim \tau^{-2} (\mathfrak{F}_1 + \tau^{-1}\mathfrak{E}_2) \mathfrak{F}_k + \tau^{-2} (\mathfrak{F}_{k-1} + \tau^{-1} \mathfrak{E}_k) \mathfrak{F}_2 \\
		+ \sum_{\ell = 1}^{k-1} \tau^{-2} \bigl( \mathfrak{F}_{\ell - 1}^{1/2} \mathfrak{F}_{\ell}^{1/2} + \tau^{-1} \mathfrak{E}_{\ell}^{1/2} \mathfrak{E}_{\ell+1}^{1/2} \bigr) \mathfrak{F}_{k - \ell}^{1/2} \mathfrak{F}_{k - \ell + 1}^{1/2}\\
		+ \sum_{j = 1}^k \sum_{\ell \leq k - j} \tau^{j-2} \bigl( \mathfrak{F}_{\ell - 1}^{1/2} \mathfrak{F}_{\ell}^{1/2} + \tau^{-1} \mathfrak{E}_{\ell}^{1/2} \mathfrak{E}_{\ell+1}^{1/2} \bigr) \mathfrak{F}_{k - j - \ell}^{1/2} \mathfrak{F}_{k - j - \ell + 1}^{1/2}.
	\end{multline}

	The estimates \eqref{eq:boxphi:gns:smallk}, \eqref{eq:boxphi:gns:bigk}, \eqref{eq:boxpsi:gns:smallk}, \eqref{eq:boxpsi:gns:bigk} together implies we can close the bootstrap using no more than 2 commutations: for $k = 0,1$ we will use the versions for small $k$; and for $k = 2$ we will use the versions for big $k$. 
	We now implement this scheme and plug in our bootstrap assumptions. We treat each of the 6 cases separately. 

	\textbf{The estimates for $\mathfrak{E}_0$.} For $\mathfrak{E}_0$ we will use \eqref{eq:boxphi:gns:smallk}, which gives
	\[ \tau^{-1/2}\|\Box\phi\|_{\Lebw_1^2} \lesssim \tau^{-2} (\mathfrak{F}_0 + \tau^{-1} \mathfrak{E}_0^{1/2} \mathfrak{E}_1^{1/2})^2 \lesssim \delta^2 \tau^{-2} \]
	by the bootstrap assumptions. Hence from \eqref{eq:energest:e} we see that the improved estimate follows. 

	\textbf{The estimates for $\mathfrak{F}_0$.} For $\mathfrak{F}_0$ we will use \eqref{eq:boxpsi:gns:smallk}, which gives 
	\[ \tau^{-1/2} \|\Box\psi\|_{\Lebw_1^2} \lesssim \tau^{-2} (\mathfrak{F}_0 + \tau^{-1} \mathfrak{E}_0^{1/2} \mathfrak{E}_1^{1/2})\mathfrak{F}_0^{1/2} \mathfrak{F}_1^{1/2} \lesssim \delta^2 \tau^{-2} \ln(\tau) \]
	by the bootstrap assumptions. Hence from \eqref{eq:energest:f} we see that the improved estimate follows. 

	\textbf{The estimates for $\mathfrak{E}_1$.} For $\mathfrak{E}_1$, \eqref{eq:boxphi:gns:smallk} gives
	\begin{multline*}
		\tau^{-1/2} \|\Box\phi\|_{\Sobw_1^{1,2}} \lesssim \tau^{-2} (\mathfrak{F}_0 + \tau^{-1} \mathfrak{E}_0^{1/2} \mathfrak{E}_1^{1/2})(\mathfrak{F}_0^{1/2} \mathfrak{F}_1^{1/2} + \tau^{-1} \mathfrak{E}_1^{1/2} \mathfrak{E}_2^{1/2}) \\
		+ \tau^{-1} (\mathfrak{F}_0 + \tau^{-1} \mathfrak{E}_0^{1/2} \mathfrak{E}_1^{1/2})^2 \lesssim \delta^2 \tau^{-2}\ln(\tau) + \delta^2 \tau^{-1}.
	\end{multline*}
	From \eqref{eq:energest:e} we see that 
	\[ \mathfrak{E}_1(\tau)^2 - \epsilon^2 \lesssim \int_2^{\tau} \delta^3 \sigma^{-1} \ln(\sigma)~\D\sigma \lesssim \delta^3 \ln(\tau)^2 \]
	and the improved estimate follows. 

	\textbf{The estimates for $\mathfrak{F}_1$.} From \eqref{eq:boxpsi:gns:smallk} with $k = 1$ we get
	\begin{multline*}
		\tau^{-1/2} \|\Box\psi\|_{\Sobw_1^{1,2}} \lesssim \tau^{-2} (\mathfrak{F}_0 + \tau^{-1} \mathfrak{E}_0^{1/2} \mathfrak{E}_1^{1/2})\mathfrak{F}_1^{1/2} \mathfrak{F}_2^{1/2} \\
		+ \tau^{-2} (\mathfrak{F}_0^{1/2} \mathfrak{F}_1^{1/2} + \tau^{-1}\mathfrak{E}_1^{1/2} \mathfrak{E}_2^{1/2}) \mathfrak{F}_0^{1/2} \mathfrak{F}_1^{1/2} + \tau^{-1} (\mathfrak{F}_0 + \tau^{-1} \mathfrak{E}_0^{1/2} \mathfrak{E}_1^{1/2}) \mathfrak{F}_0^{1/2} \mathfrak{F}_1^{1/2}
	\end{multline*}
	Plugging in the bootstrap assumptions we get
	\[ \lesssim \delta^2 \tau^{-2} \ln(\tau) \tau^{1/2} \ln(\tau)^{1/2} + \delta^2 \tau^{-2} \ln(\tau)^2 + \delta^2 \tau^{-1} \ln(\tau) \lesssim \delta^2 \tau^{-1} \ln(\tau).\]
	This means by \eqref{eq:energest:f} we get
	\[ \mathfrak{F}_1(\tau)^2 - \epsilon^2 \lesssim \int_2^{\tau} \delta^3 \sigma^{-1} \ln(\sigma)^3 ~\D\sigma \lesssim \delta^3 \ln(\tau)^4\]
	and the improved estimate for $\mathfrak{F}_1$ follows. 

	\textbf{The estimates for $\mathfrak{E}_k$, where $k \geq 2$.}
	For the higher order estimates we will use \eqref{eq:boxphi:gns:bigk}. 
	We will use the very rough estimate that 
	\[ \mathfrak{E}_k \leq \delta \tau^{\max(k-1,0)}\ln(\tau), \qquad \mathfrak{F}_k \leq \delta \tau^{\max(k-1,0)}\ln(\tau)^2,\]
	and only be very careful about the cases where $j = k$. 
	This gives
	\begin{multline}
		\tau^{-1/2} \|\Box\phi\|_{\Sobw_1^{k,2}} \lesssim \delta^2 \tau^{-2} \ln(\tau)^2 \tau^{k-2} \ln(\tau)^2 \\
		+ \sum_{\ell = 1}^{k-1} \delta^2 \tau^{-2}  \tau^{\frac12(\max(\ell-2,0) + \max(\ell-1,0))} \ln(\tau)^2 \tau^{\frac12 ( \max(k-\ell - 2,0) + \max(k-\ell-1,0))} \ln(\tau)^2\\
		+ \sum_{j = 1}^{k-1} \sum_{\ell \leq k - j} \delta^2 \tau^{j-2} \tau^{\frac12(\max(\ell-2,0)+\max(\ell-1,0))}\ln(\tau)^2 \tau^{\frac12(\max(k - j - \ell - 2,0) + \max(k -j - \ell - 1,0))} \ln(\tau)^2 \\
		+ \tau^{k-2} (\mathfrak{F}_0 + \tau^{-1} \mathfrak{E}_0^{1/2} \mathfrak{E}_1^{1/2})^2.
	\end{multline}
	We next note (since $k,\ell$ are integers and we restricted $\ell \in [1,k-1]$, and $k \geq 2$)
	\begin{multline*}
		\max(\ell - 2,0) + \max(\ell-1,0) + \max(k - \ell - 2,0) + \max (k - \ell - 1,0) \\
		= \max(2k - 6,2k - 2\ell - 3, 2\ell - 3, 0) \leq 2k-4.
	\end{multline*}
	Similarly (now $j \in [1,k-1]$ and $\ell \in [0,k-j]$)
	\begin{multline*}
		\max(\ell - 2,0) + \max(\ell-1,0) + \max(k-j-\ell -2,0) + \max(k - j - \ell - 1,0) + 2j \\
		= \max( 2k - 2\ell - 3, 2k - 6, 2\ell +2j - 3, 2j) \leq 2k-2.
	\end{multline*}
	So 
	\begin{equation}
		\tau^{-1/2} \|\Box\phi\|_{\Sobw_1^{k,2}} \lesssim \delta^2 \tau^{k-4} \ln(\tau)^4 + \delta^2 \tau^{k-3} \ln(\tau)^4 + \delta^2 \tau^{k-2}.
	\end{equation}
	This implies that in \eqref{eq:energest:e} we see
	\[
		\mathfrak{E}_k(\tau)^2 - \epsilon^2 \lesssim \int_2^\tau \delta^3 \sigma^{2k-3} ~\D\sigma \lesssim \delta^3 \tau^{2k-2}
	\]
	and the improved estimate follows. 

	\textbf{The estimates for $\mathfrak{F}_k$, where $k \geq 2$.}
	Finally we apply \eqref{eq:boxpsi:gns:bigk}. Again when $j < k$ we will estimate very roughly. This gives
	\begin{multline}
		\tau^{-1/2} \|\Box\psi\|_{\Sobw_1^{k,2}} \lesssim \delta^2 \tau^{-2} \ln(\tau)^2 \tau^{k-1} \ln(\tau)^2 + \delta^2 \tau^{-2} \tau^{k-2} \ln(\tau)^2 \tau \ln(\tau) \\
		+ \sum_{\ell = 1}^{k-1}\delta^2 \tau^{-2} \tau^{\max(\ell-3/2,0)} \ln(\tau)^2 \tau^{\max(k - \ell -1/2,0)} \ln(\tau)^2 \\
		+ \sum_{j = 1}^{k-1} \sum_{\ell \leq k - j} \delta^2 \tau^{j-2} \tau^{\max(\ell -3/2,0)} \ln(\tau)^2 \tau^{\max(k - j -\ell - 1/2,0)} \ln(\tau)^2 \\
		+ \tau^{k-2} \bigl( \mathfrak{F}_{ - 1}^{1/2} \mathfrak{F}_{0}^{1/2} + \tau^{-1} \mathfrak{E}_{0}^{1/2} \mathfrak{E}_{1}^{1/2} \bigr) \mathfrak{F}_{0}^{1/2} \mathfrak{F}_{1}^{1/2}.
	\end{multline}
	A similar analysis to before shows that the first three terms can be bounded by $\delta^2 \tau^{k - 5/2} \ln(\tau)^4$. The final term however is bounded by $\delta^2 \tau^{k-2} \ln(\tau)$. Hence inserting into \eqref{eq:energest:f} we see
	\[ \mathfrak{F}_k^2(\tau) - \epsilon^2 \lesssim \int_2^\tau \delta^3 \sigma^{2k-3} \ln(\sigma)^2 ~\D\sigma \lesssim \delta^3 \tau^{2k-2} \ln(\sigma)^2.\]
	In the final inequality we used Lemma \ref{lem:tech}. This implies that the improved estimates are obtained, and our theorem is proved. 
\end{proof}

\begin{rmk}
	Our results are compatible with \emph{boundedness} of generic higher derivatives of the solution. Indeed, using \eqref{Morrey} we see that $\|\partial_t \phi\|_{L^\infty(\Sigma_\tau)} \lesssim \frac{1}{\tau} \mathfrak{E}_2(\tau)$, and $\|L^i \phi\|_{L^\infty(\Sigma_\tau)} \lesssim \mathfrak{E}_2(\tau)$. The latter implies that $\|\partial_{x^i}\phi\|_{L^\infty(\Sigma_\tau)} \lesssim (\frac{1}{t} + \frac{x^i}{t\tau}) \mathfrak{E}_2(\tau)$ which is bounded. Similarly for higher derivatives. On the other hand, we have improved decay for the derivative tangential to the travelling background. Specifically, we have $\|L^1 \phi\|_{L^\infty(\Sigma_\tau)} \lesssim \mathfrak{F}_1(\tau)$. If we were to also analyze the equation satisfied by $\partial_t\phi$, we would find that $|\partial_t \phi + \partial_{x^1}\phi|$ decays like $\ln(\tau)^2 / t$. 

	In particular, our results are compatible with a \emph{lack of peeling}, where all higher derivatives exhibit a ``decay rate'' that is the same as the first derivatives of the free wave equation in dimension 2, \emph{insofar as provable by using only the $\partial_t$ energy}. (Recall from \cite{Wong2017} that to get improved interior decay one should use instead the energy corresponding to the Morawetz $K$ multiplier.) 
	If one were to wish to obtain point-wise decay of the solution and its derivatives in $d = 2$ for this problem (such as what one would need to study the quasilinear problem), one is certainly bound to use the Morawetz energy instead of the $\partial_t$-energy as given above. 
	The linear estimates described in Section \ref{sect:lin:est} remain useful in such a setting, as the Morawetz energy still controls $\Lebw_{-1}^2$ integrals of the solution, just with different $\tau$ weights. 
\end{rmk}

\bibliographystyle{amsalpha}
\bibliography{WM.bib}

\end{document}